\documentclass[preprint]{article}
\usepackage{arxiv}

\usepackage{amsmath}
\usepackage{amsthm}
\numberwithin{equation}{section}
\newtheorem{theorem}{Theorem}[section]
\newtheorem{lemma}[theorem]{Lemma}      
\newtheorem{corollary}[theorem]{Corollary}       

\newtheorem{remark}[theorem]{Remark}
\newtheorem{definition}[theorem]{Definition}
\newtheorem{assumption}[theorem]{Assumption}
\usepackage{amssymb}
\usepackage{amsfonts,mathrsfs,epstopdf}
\usepackage[utf8]{inputenc} 
\usepackage[T1]{fontenc}    
\usepackage{url}            
\usepackage{booktabs}       
\usepackage{amsfonts}       
\usepackage{nicefrac}       
\usepackage{microtype}      
\usepackage{lipsum}

\usepackage{graphicx} 
\usepackage{subcaption}
\usepackage{caption}
\usepackage{listings}
\usepackage{pgfplots}
\usepackage{tikz}

\usepackage[backend=biber,style=numeric]{biblatex}
\usepackage{hyperref} 
\hypersetup{
	colorlinks=true,
	linkcolor=blue,
	citecolor=blue,
	urlcolor=blue
}
\usetikzlibrary{patterns}
\DeclareBibliographyCategory{cited}
\usepackage{setspace}
\usepackage{titlesec}
\usepackage{indentfirst}
\title{Computational Methods and Verification Theorem for Portfolio-Consumption Optimization under Exponential O-U Dynamics}

\author{
 Zhaoxiang Zhong \\
  Faculty of Science\\
  National University of Singapore\\
  \texttt{e1349203@u.nus.edu} \\
   \And
Haiming Song \\
  School of Mathematics\\
 Jilin University\\
  \texttt{songhaiming@jlu.edu.cn} \\
}

\begin{document}
\maketitle
\begin{abstract}
In this paper, we focus on the problem of optimal portfolio-consumption policies in a multi-asset financial market, where the $n$ risky assets follow Exponential Ornstein-Uhlenbeck processes, along with one risk-free bond. The investor's preferences are modeled using Constant Relative Risk Aversion utility with state-dependent stochastic discounting. The problem can be formulated as a high-dimensional stochastic optimal control problem, wherein the associated value function satisfies a Hamilton-Jacobi-Bellman (HJB) equation, which constitutes a necessary condition for optimality. We apply a variable separation technique to transform the HJB equation to a system of ordinary differential equations (ODEs). Then a class of hybrid numerical approaches that integrate exponential Rosenbrock-type methods with Runge-Kutta methods is proposed to solve the ODE system. More importantly, we establish a rigorous verification theorem that provides sufficient conditions for the existence of value function and admissible optimal control, which can be verified numerically. A series of experiments are performed, demonstrating that our proposed method outperforms the conventional grid-based method in both accuracy and computational cost. Furthermore, the numerically derived optimal policy achieves superior performance over all other considered admissible policies.
\end{abstract}


\section{Introduction}
		\par Optimal portfolio-consumption problem is one of the most important topics in mathematical finance. Merton \cite{MERTON1971373} proposes the fundamental stochastic control model to address this problem, where the stock price follows geometric Brownian motion. His work has attracted significant attention, leading to the development of numerous models and approaches aimed at adapting to the stochastic nature of real-world markets. In particular, various mean-reverting models are introduced to the problem, driven by the discovery of mean-reverting behavior in asset prices \cite{CAMPBELL1987373,cc65fdf1-38c7-353c-aa33-988823769d7e,POTERBA198827}, such as Ornstein-Uhlenbeck (O-U) model \cite{62dfb375-f276-343b-afce-284aaa9740e1}, exponential O-U model \cite{Benth01072005} and Cox-Ingersoll-Ross model \cite{doi:10.1142/S0219024909005452}. In this paper, we will consider exponential O-U model, which effectively describes and simulates the general mean-reverting dynamics over time of certain asset prices.
		\par A variety of utility functions have been integrated into the portfolio-consumption optimization framework, among which the Constant Absolute Risk Aversion (CARA) utility \cite{Ma03072019,https://doi.org/10.1155/2014/153793} and the Constant Relative Risk Aversion (CRRA) utility \cite{KASSIMATIS2021101932,miao2023optimalinvestmentconsumptionstrategies} are the most prevalent ones. The CARA framework has been widely studied for its analytical tractability; however, it assumes constant absolute risk aversion, which implies a fixed allocation to risky assets regardless of wealth—an assumption that contradicts empirical evidence showing wealth-dependent risk preferences. In contrast, CRRA utility provides a more realistic framework, as it captures how investment behavior varies with changes in wealth \cite{https://doi.org/10.1002/hec.1331}. Moreover, recent studies have emphasized CRRA utility formulations featuring stochastic discounting, including non-exponential discounting \cite{ekeland2008investment}, state-dependent discounting \cite{doi:10.1142/S0219024909005452}, and regime-switching discounting \cite{PIRVU2014142}. The state-dependent discounting setting, in which the discount rate explicitly depends on the underlying stock price dynamics, allows the model to better capture the changing risk preferences of investors in financial markets. Overall, a model that incorporates exponential O–U dynamics and CRRA utility with state-dependent stochastic discounting provides a more realistic and flexible framework for addressing real-world financial optimization problems.
		\par In terms of solutions, researchers have proposed a wide range of techniques to analyze optimal portfolio-consumption problems. Early contribution by Karatzas \cite{doi:10.1137/0325086} utilizes martingale methods to derive weak solutions. With the introduction of Hamilton-Jacobi-Bellman (HJB) equations, many studies have reformulated the problem into a HJB-kind nonlinear partial differential equation (PDE) and derived closed-form solutions in $1$-dimensional case \cite{LEACH2007368,pham2002smooth,zariphopoulou1999optimal}. To tackle more general cases, grid-based numerical schemes have been introduced, including Markov chain approximation combined with logarithmic transformation \cite{5991218}, fixed-point iterative methods \cite{Berdjane2013Optimal}, and hybrid Monte Carlo with finite difference approaches \cite{TSAI2018170}. However, multi-asset problems involving CRRA utility functions give rise to nonlinear high-dimensional HJB equations that are typically intractable for closed-form solutions. Moreover, conventional grid-based numerical methods suffer from the curse of dimensionality when applied to such problems. To address this computational challenge, various numerical methods have been proposed, such as sparse grid methods \cite{doi:10.1137/110834950}, Hopf formulas \cite{darbon2016algorithmsovercomingcursedimensionality}, and deep learning methods \cite{doi:10.1137/23M155832X}. These works provide a good approximation of the high-dimensional value function, however, to the best of our knowledge, their numerical implementation can often be complex and lack rigorous verification theorems to ensure optimality. 
		\par In this paper, we also adopt the framework of deriving the HJB equation for the optimal control problem. Then we apply a variable separation technique to transform the high-dimensional HJB equation into an ODE system, significantly reducing computational complexity. To solve the resulting ODEs numerically, we develop hybrid numerical approaches that integrate exponential Rosenbrock-type methods \cite{doi:10.1137/080717717,LI2021113360} with Runge-Kutta methods, which demonstrate superior accuracy and efficiency. Most importantly, we derive a rigorous verification theorem to validate the optimality of our approach, which is an aspect often overlooked in previous numerical methods.
		\par The remainder of this paper is organized as follows: In Section \ref{sec2}, we start with introducing the corresponding modeling framework of our paper. In Section \ref{sec3}, the HJB equation corresponding to the stochastic optimal control problem is derived, along with the form of the optimal policy. Then a variable separation technique is applied to transform the HJB equation into an ODE system. In Section \ref{sec4}, we introduce the numerical methods to solve the ODE system and present a convergence theorem. A rigorous verification theorem is derived in Section \ref{sec5}. In Section \ref{sec6}, we conduct a series of numerical experiments to evaluate the performance of our algorithms. Finally, we summarize our findings and present conclusions in Section \ref{sec7}.

\section{Problem formulation}
		\label{sec2}
        In this section, we introduce the modeling framework and formulate the stochastic optimal control problem including the constraint equations and the optimization objective. 
		\subsection{Constraint Equations}
		In our paper, we extend the assumptions of the financial market model in \cite{Benth01072005} to the $n$-dimensional setting.
		\begin{assumption}
			The financial market under consideration is assumed to be arbitrage-free over a fixed time horizon $[0,T]$, where $0\leq T<\infty$. Let $(\Omega,\mathcal{F},\mathbb{P},\left\{\mathcal{F}_{t}\right\}_{0\leq t \leq T})$ be a complete filtered probability space satisfying usual conditions, and $\left\{B_{t}=\left(B_{t}^{(i)}\right)_{1\leq i \leq n}^{\top}, 0\leq t \leq T\right\}$ denote the $n$-dimensional Brownian motion defined on it. The market consists of $n+1$ tradable assets: one risk-free bond and $n$ risky stocks. Within this framework, a self-financed small investor who allocates wealth across these assets is considered.
            \label{a2.1}
		\end{assumption}
		\begin{assumption}
			The price $P_{0}(t)$ of bond is governed by following ODE:
			\begin{equation}
				\left\{
				\begin{aligned}
					\centering
					\mathrm{d}P_{0}(t)&=rP_{0}(t)\,\mathrm{d}t,\quad 0< t\leq T,\\
					P_{0}(0) &=p_{0},
				\end{aligned} \right.
				\notag
			\end{equation}
			where $p_{0}>0$ is the initial price, $r>0$ is the constant interest rate. The price $P_{i}(t)$ of $i$th ($1\leq i\leq n$) stock follows exponential O-U process
			\begin{equation}
				\left\{
				\begin{aligned}
					\centering
					\mathrm{d}P_{i}(t)&=\alpha_{i}(\mu_{i}-\ln P_{i}(t))P_{i}(t)\,\mathrm{d}t+\sigma_{i}P_{i}(t)\, \mathrm{d}B_{t},\quad 0< t\leq T,\\
					P_{i}(0) &=p_{i},
				\end{aligned} \right.
				\notag
			\end{equation}
			where $p_{i}>0$ is the initial price, $\alpha_{i},\mu_{i}>0$ are constants, $\sigma=(\sigma_{ij})_{1\leq i,j\leq n}$ is the non-degenerate volatility matrix whose $i$th row is vector $\sigma_{i}$ and entries are all positive.
            \label{a2.2}
		\end{assumption}
		\begin{assumption}
			A portfolio-consumption pair $(\pi(t), C(t))$ is defined, where $\pi(t) = (\pi_{i}(t))_{1 \leq i \leq n}^{\top}$ denotes the allocation of wealth across the $n$ stocks, with $\pi_{i}(t)$ representing the amount invested in the $i$th stock at time $t$. $C(t)$ denotes the consumption rate at time $t$. In addition, $\pi_0(t)$ represents the amount of wealth invested in the bond at time $t$. Assume $\pi(t)$ and $C(t)$ are measurable and adapted to the filtration $\left\{\mathcal{F}_{t}\right\}$, with $C(t) \in [0,+\infty),\pi_{i}(t) \in \mathbb{R}$, for all $t\in [0,T]$. Note that for $i = 0, 1, \ldots, n$, $\pi_i(t)$ may take negative values, indicating short-selling of the corresponding bond or stock. 
            \label{a2.4}
		\end{assumption} 
		\begin{remark}
			For analytical convenience, we let $S(t)=(S_{i}(t))_{1\leq i\leq n}^{\top}$, where $S_{i}(t)=\ln P_{i}(t)$. By Ito equation \cite{doi:10.1137/0325086}, we have 
        			\begin{equation}
				\left\{
				\begin{aligned}
					\centering
					\mathrm{d}S(t)&=\operatorname{diag}(\alpha)(w-S(t))\,\mathrm{d}t+\sigma \, \mathrm{d}B_{t},\quad 0< t\leq T,\\
					S(0) &=S_0,
				\end{aligned} \right.
				\label{2.0}
			\end{equation}
			where 
            \begin{equation*}
                \operatorname{diag}(\alpha)=
		\begin{pmatrix}
			\alpha_{1}& \cdots &   0  \\
			\vdots          & \ddots & \vdots \\
			0          & \cdots & \alpha_{n}
             \end{pmatrix},
            \hspace{2em}w=\left(\mu_{i}+\frac{\|\sigma_{i}\|^{2}}{2\alpha_{i}}\right)_{1\leq i \leq n}^{\top}.
            \end{equation*}
		\end{remark}
		
         Let $X(t)$ denotes the wealth at time $t\in [0,T]$. The amount of wealth allocated to the bond is then given by $\pi_{0}(t)\equiv X(t)-\sum_{i=1}^{n}\pi_{i}(t)$. Building on equation (\ref{2.0}) and Ito equation, we can readily derive that $X(t)$ has following dynamics:
        \begin{equation}
				\left\{
				\begin{aligned}
					\centering					\mathrm{d}X(t)&=rX(t)\,\mathrm{d}t+\pi^{\top}(t)(\operatorname{diag}(\alpha)(\mu-S(t))-re)\,\mathrm{d}t-C(t)\,\mathrm{d}t+\pi^{\top}(t)\sigma \, \mathrm{d}B_{t},\quad 0< t\leq T,\\
					X(0) &=X_0,
				\end{aligned} \right.
				\label{2.1}
			\end{equation}
		where $e$ is $n$-dimensional column vector of ones and $\mu=(\mu_{i})_{1\leq i \leq n}^{\top}$.
		
\subsection{Optimization Objective}
		To facilitate, define portfolio process $\pi\triangleq\{\pi(t),0\leq t\leq T\}$, consumption process $C\triangleq\{C(t),0\leq t\leq T\}$ and corresponding wealth process $X\triangleq\{X(t),0\leq t\leq T\}$. An optimization objective functional is stated as
		\begin{equation}		J(\pi,C)\triangleq\mathbb{E}\left[\int_{0}^{T}U(t,C(t))\,\mathrm{d}t+U(T,X(T))\right].
            \label{2.100}
		\end{equation}
		 The utility function $U(t,c)$ represents the CRRA utility  with a stochastic discount factor that depends on the stock price. Specifically,
		  \begin{equation*}
		      U(t,c)\triangleq\gamma^{-1}\psi(t)^{1-\gamma}c^{\gamma}, \quad \gamma\in(0,1),
		  \end{equation*}
		where 
		\begin{equation}
			\psi(t)\triangleq\exp\left\{-\frac{1}{1-\gamma}\int_{0}^{t}[\rho_{0}+S(u)^{\top}\rho+S(u)^{\top}\varrho S(u)] \,\mathrm{d}u\right\}, ~ \rho_{0}\in \mathbb{R}, ~ \rho\in \mathbb{R}^{n}, ~ \varrho\in \mathbb{S}^{n}.
            \label{psi}
		\end{equation}
		  Here, $\mathbb{S}^{n}$ denotes the set of all symmetric $n\times n$ matrices. Next we introduce the admissible set.
		\begin{definition}
			 A portfolio-consumption pair process $(\pi,C)$ is said to be admissible for any initial wealth $X_0$, if and only if $(\pi,C)$ satisfy following conditions:
			\begin{enumerate}
				\item[(a)]
				$\mathbb{E}\left[\displaystyle\int_{0}^{T} C(t)\,\mathrm{d}t\right] < +\infty ,$
				\item[(b)]$\mathbb{E}\left[\displaystyle\int_{0}^{T}\|\pi(t)\|^{2}\,\mathrm{d}t\right]<+\infty  ,$
				\item[(c)]				$\mathbb{E}\left[|J(\pi,C)|\right]<+\infty$,
				\item[(d)]
				$X(t)\geq 0, \quad \forall t\in[0,T]$.
			\end{enumerate}
			Here, $\| \cdot \|$ denotes the Frobenius norm, and we will adopt this notation throughout. The class of all such pairs is denoted by $\mathcal{A}_{X_0}$.
            \label{d2.5}
		\end{definition}
        \begin{remark}
			Under the conditions (a)-(b) in Definition \ref{d2.5}, it guarantees that SDE (\ref{2.1}) has unique strong solution
				\begin{align}
					X(t)=\exp\{rt\}&\left\{X_0+\int_{0}^{T}\exp\left\{-ru\right\}\left[\pi(u)^{\top}(\operatorname{diag}(\alpha)(\mu-S(u)-re)-C(u)\right]\,\mathrm{d}u  \right. \notag \\
					&\left.\quad +\int_{0}^{T}\exp\left\{-ru\right\}\pi(u)^{\top}\sigma \mathrm{d}B_{u}\right\}.
                    \label{ss} 
				\end{align}
		\end{remark}
		Our aim is to find the optimal portfolio-consumption pair process $(\pi^{*},C^{*})$ within $\mathcal{A}_{X_0}$ that maximizes (\ref{2.100}) subject to the constraint (\ref{2.1}).  

        \section{HJB Equation and Corresponding ODE System}
        \label{sec3}
        In this section, we derive the HJB equation corresponding to stochastic optimal control problem (\ref{2.1})-(\ref{2.100}). We also present the optimal policy, which can be expressed in a feedback form through the value function. Moreover, a variable separation technique \cite{Benth01072005,doi:10.1142/S0219024909005452} including power transformation and Feynman-Kac representation formula \cite{Karatzas1998Brownian} is applied to transform the HJB equation to an ODE system.
        \subsection{HJB Equation and Optimal Policy}
        To solve the optimal control problem in a continuous sense, value function is defined as the maximum utility attainable starting from time $t$, wealth $X(t) \equiv x$, logarithm stock price $S(t) \equiv S$ and discount factor $\psi(t) \equiv \psi $. Specifically,
		\begin{equation}			V(t,x,S,\psi)\triangleq\mathop{\sup}\limits_{(\pi,C)\in\mathcal{A}_{x}}\mathbb{E}^{(t,x,S,\psi)}\left[\int_{t}^{T}U(u,C(u))\,\mathrm{d}u+U(T,X(T))\right].
        \label{o3}
		\end{equation} 
        Here $\mathbb{E}^{(t,x,S,\psi)}$ denotes the expectation from initial state $(t,x,S,\psi)$. Suppose that $V(t, x, S, \psi) \in C^{1,2}([0,T] \times \mathbb{R} \times \mathbb{R}^n \times \mathbb{R}_+)$, where $\mathbb{R}_+$ denotes the set of all positive real numbers and $ C^{1,2}$ denotes the class of functions that are once continuously differentiable in time and twice continuously differentiable in all state variables. Then $V$ is a classical solution to the HJB equation (see \cite[Proposition 3.5 of Chapter 4]{yong1999stochastic}), which, in the context of our problem, is given as
           \begin{equation}
               \frac{\partial v}{\partial t}
+ r x\,\frac{\partial v}{\partial x}
+ \beta \psi\,\frac{\partial v}{\partial \psi}
+ \mathcal{L}_{S}v    
+ \mathop{\sup}\limits_{(\pi,C)\in\mathcal{A}_{x}} \left\{\mathcal{G}_{xS \psi}\,v\!\left(\pi,C\right)\right\}
=0, \label{2.5}                    
           \end{equation}
           where
			\begin{align*}
				\mathcal{L}_{S}v=&(w^{\top}-S^{\top})\operatorname{diag}(\alpha)D_{S}v+\frac{1}{2}\mathrm{tr}\left\{D_{SS}^{2}v\sigma\sigma^{\top}\right\}, \\
        \mathcal{G}_{xS \psi}\,v\!\left(\pi,C\right)=&\gamma^{-1}\psi^{1-\gamma}C^{\gamma}+\frac{\partial v}{\partial x}\pi^{\top}\left[\operatorname{diag}(\alpha)(\mu-S)-re\right]\\
        &-\frac{\partial v}{\partial x}C
			+\frac{1}{2}\frac{\partial^{2} v}{\partial x^{2}}\pi^{\top}\sigma\sigma^{\top}\pi+\pi^{\top}\sigma\sigma^{\top}D_{xS}^{2}v,
			\end{align*}
            and 
            \begin{equation*}
                D_{S}v=(\frac{\partial v}{\partial S_{i}})_{1\leq i \leq n}^{\top},\quad
		D_{SS}^{2}v=(\frac{\partial^{2} v}{\partial S_{i}\partial S_{j}})_{1\leq i,j \leq n},\quad\beta=-\frac{1}{1-\gamma}[\rho_{0}+S^{\top}\rho+S^{\top}\varrho S],
            \end{equation*}
		with terminal condition
		\begin{equation}
			v(T,x,S,\psi)=U(T,x).
			\label{2.51}
		\end{equation}
		
		\par We denote the solution to the HJB equation by $v$ to distinguish it from the value function $V$, as the solution may represent a broader class of functions that includes $V$. HJB equation (\ref{2.5})-(\ref{2.51}) constitutes as a necessary condition for value function. A rigorous verification theorem, which provides sufficient conditions under which a candidate function $v$ coincides with the value function $V$, will be derived in Section \ref{sec4}. When both the necessary and sufficient conditions are fulfilled, the candidate function $v$ is indeed the value function $V$.
        \par As demonstrated in Chapter 4 of \cite{yong1999stochastic}, if $V\in C^{1,2}([0,T]\times\mathbb{R}\times \mathbb{R}^{n} \times\mathbb{R})$ and the sufficient conditions for optimality are satisfied, solving for the supremum yields optimal policy $(\pi^{*},C^{*})$. Therefore, we now focus on solving PDE (\ref{2.5})-(\ref{2.51}) to obtain a candidate solution. To reduce the complexity of the HJB equation, we consider applying a variable separation technique. Following the idea in \cite{zariphopoulou1999optimal}, equation (\ref{2.51}) motivates the following power transformation ansatz:
		\begin{equation}
			v(t,x,S,\psi)=\gamma^{-1}\psi^{1-\gamma}x^{\gamma}\varphi^{1-\gamma}(t,S),
			\label{2.8}
		\end{equation}
		where $\varphi(t,S)\in C^{1,2}([0,T]\times \mathbb{R}^{n})$.
		By substituting formula (\ref{2.8}) into equation (\ref{2.5})-(\ref{2.51}) and solving the supremum in (\ref{2.5}), we obtain the PDE satisfied by $\varphi(t,S)$ as
			\begin{align}
				\frac{\partial \varphi}{\partial t}&+(w^{\top}-S^{\top})\operatorname{diag}(\alpha)D_{S}\varphi+\frac{\gamma}{1-\gamma}\left[(\mu^{\top}-S^{\top})\operatorname{diag}(\alpha)-re^{\top}\right]D_{S}\varphi \notag   \\
				&+\left\{\frac{r\gamma}{1-\gamma}+\beta+\frac{\gamma\left[(\mu^{\top}-S^{\top})\operatorname{diag}(\alpha)-re^{\top}\right](\sigma\sigma^{\top})^{-1}\left[\operatorname{diag}(\alpha)(\mu-S)-re\right]}{2(1-\gamma)^{2}}\right\}\varphi  \notag \\
                &+\frac{1}{2}\mathrm{tr}\left\{D_{SS}^{2}\varphi\sigma\sigma^{\top}\right\}+1=0,
                \label{2.9}
			\end{align}
		with terminal condition
		\begin{equation}
			\varphi(T,S)=1.
			\label{2.10}
		\end{equation}
        In the next step, we only have to solve PDE (\ref{2.9})-(\ref{2.10}).  
        \begin{remark}
           Following the approach of \cite{Benth01072005}, once PDE (\ref{2.9})-(\ref{2.10}) has been solved and the sufficient conditions for optimality are satisfied, one can derive the optimal policy in feedback form:  
            \begin{align}
         \pi^*(t, X^*(t), S(t)) &= \frac{1}{1 - \gamma} X^*(t)(\sigma \sigma^\top)^{-1} \left[ \operatorname{diag}(\alpha)(\mu - S(t)) - re \right] \notag \\
         & \hspace{2em}+ X^*(t) D_S \varphi(t, S(t)) \varphi^{-1}(t, S(t)), 
         \label{3.221}  \\
          C^*(t, X^*(t), S(t)) &= \varphi^{-1}(t, S(t)) X^*(t). \label{3.222}
            \end{align}
         \label{rr3.1}
        \end{remark}
         \subsection{ODE System}
         \label{sec3.2}
        To derive an explicit form for $\varphi$, we adopt the change of measure approach introduced in \cite{Benth01072005}. Reviewing (\ref{2.1}), if Novikov condition
		\begin{equation}
			\mathbb{E}\left[\exp\left\{\frac{1}{2}\int_{0}^{T}\frac{\gamma^{2}}{(1-\gamma)^{2}}\|\sigma^{-1}(\operatorname{diag}(\alpha)(\mu-S(t))-re)\|^{2}\,\mathrm{d}t\right\}\right]<+\infty
			\label{2.11}
		\end{equation}
        is satisfied, by Gisanov Theorem \cite{Karatzas1998Brownian}, we can define a new probability measure $\bar{\mathbb{P}}$ under which the process $S(t)$ satisfies
		\begin{equation}
			\mathrm{d}S(t)=\operatorname{diag}(\alpha)(w-S(t))\,\mathrm{d}t+\frac{\gamma}{1-\gamma}\left[\operatorname{diag}(\alpha)(\mu-S(t))-re\right]\,\mathrm{d}t+\sigma \mathrm{d}\bar{B}_{t}.
			\label{Stfk}
		\end{equation}
		We note that condition (\ref{2.11}) is essential as it will be incorporated into the verification theorem in Section \ref{sec4}. Assume condition (\ref{2.11}) holds,
		we define
		\begin{equation*}
			H(S)=\frac{r\gamma}{1-\gamma}+\beta+\frac{\gamma\left[(\mu^{\top}-S^{\top})\operatorname{diag}(\alpha)-re^{\top}\right](\sigma\sigma^{\top})^{-1}\left[\operatorname{diag}(\alpha)(\mu-S)-re\right]}{2(1-\gamma)^{2}}.
		\end{equation*}
		Based on equation (\ref{Stfk}), by Feynman-Kac representation Formula, we obtain expectation representation of $\varphi$ as
		\begin{align}
			\varphi(t,S)&= \bar{\mathbb{E}}^{(t,x,S,\psi)}\left[\exp\left\{\int_{t}^{T}H(S(\tau))\,\mathrm{d}\tau\right\}+\int_{t}^{T}\exp\left\{\int_{t}^{u}H(S(\tau))\,\mathrm{d}\tau\right\}\,\mathrm{d}u \right] \notag \\
            &\triangleq \varphi_{1}(t,S)+\varphi_{2}(t,S). \label{phijia}
		\end{align}
		Subsequently, we only have to solve $\varphi_1,\varphi_2$ to obtain $\varphi$.
	    \begin{lemma}
	    	$\varphi_{2}(t,S)$ can be written as 
	    	\begin{equation}
	    		\varphi_{2}(t,S)=\int_{t}^{T}\varphi_{1}(u,S)\,\mathrm{d}u.   
                \label{3phi}
	    	\end{equation}
	    \end{lemma}
	    \begin{proof}
	    	First, we define $\tilde{\varphi_{1}}(t,\tau,S)$ as the function $\varphi_{1}(t,S)$ with the fixed terminal time $T$ replaced by a variable $\tau \in (t,T]$. We can derive that 
	    	\begin{align}
	    		\varphi_{2}(t,S)&=\bar{\mathbb{E}}^{(t,x,S,\psi)}\left[\int_{t}^{T}\exp\left\{\int_{t}^{\tau}H(S(u))\,\mathrm{d}u\right\}\,\mathrm{d}\tau\right] \notag \\    		&=\int_{t}^{T}\bar{\mathbb{E}}^{(t,x,S,\psi)}\left[\exp\left\{\int_{t}^{\tau}H(S(u))\,\mathrm{d}u\right\}\right]\,\mathrm{d}\tau \notag \\
	    		&=\int_{t}^{T}\tilde{\varphi_{1}}(t,\tau,S)\,\mathrm{d}\tau.
	    		\label{2.14}
	    	\end{align}
	    Since both the drift and diffusion terms in SDE (\ref{2.0}), as well as the function $H$, are independent of time $t$, the system $\tilde{\varphi}_1$ is time-consistent.
          Hence, there exists a function $\phi(\tau -t ,S)$ satisfying
	    \begin{equation*}
	    	\phi(\tau -t,S)=\tilde{\varphi_{1}}(t,\tau,S).
	    \end{equation*}
	     By equation (\ref{2.14}) and setting $T-u=\tau-t$, we have
	    \begin{equation*}
	    	\varphi_{2}(t,S)
	    	=\int_{t}^{T}\phi(\tau -t,S)\,\mathrm{d}\tau = \int_{t}^{T}\phi(T -u,S)\,\mathrm{d}u=\int_{t}^{T}\tilde{\varphi_{1}}(u,T,S)\,\mathrm{d}u=\int_{t}^{T}\varphi_1(u,S)\,\mathrm{d}u.
	    \end{equation*}
	    \end{proof}
		\par Now we only have to solve $\varphi_{1}(t,S)$.
        By Feynman-Kac Formula,  the conditional expectation representation of $\varphi_{1}$ can be equivalently formulated as the following PDE:
		\begin{align}
			\frac{\partial \varphi_{1}}{\partial t} 
			&+ (w^{\top} - S^{\top}) \operatorname{diag}(\alpha) D_{S} \varphi_{1} 
			+ \frac{\gamma}{1-\gamma} 
			\bigl[ (\mu^{\top} - S^{\top}) \operatorname{diag}(\alpha) - r e^{\top} \bigr] D_{S} \varphi_{1} \notag  \\
			&+ \biggl\{ \frac{\gamma \bigl[ (\mu^{\top} - S^{\top}) \operatorname{diag}(\alpha) - r e^{\top} \bigr] 
				(\sigma \sigma^{\top})^{-1} \bigl[ \operatorname{diag}(\alpha) (\mu - S) - r e \bigr]}{2(1-\gamma)^{2}} \notag\\
                & +\frac{1}{1-\gamma} 
			( r \gamma - \rho_{0} - S^{\top} \rho - S^{\top} \varrho S )
			\biggr\} \varphi_{1}+ \frac{1}{2} \operatorname{tr} \biggl\{ D_{SS}^{2} \varphi_{1} \sigma \sigma^{\top} \biggr\} = 0,
			\label{2.17}
		\end{align}
        with terminal condition
		\begin{equation*}
			\varphi_1(T,S)=1.
		\end{equation*}
        Note that $H(S)$ is a quadratic function of $S$, we assume a solution ansatz for $\varphi_{1}(t,S)$ of this form \cite{Benth01072005}: 
		\begin{equation}
			\varphi_{1}(t,S)=\exp\left\{S^{\top}g(t)S+S^{\top}f(t)+f_{0}(t)\right\}.
			\label{2.15}
		\end{equation}
		Here 
        \begin{equation*}
            g(t)=(g_{ij}(t))_{n\times n}:[0,T]\rightarrow \mathbb{R}^{n\times n},\quad f(t)=(f_{i}(t))_{1 \leq i \leq n}^{\top}:[0,T]\rightarrow \mathbb{R}^{n},\quad f_{0}(t):[0,T]\rightarrow \mathbb{R}.
        \end{equation*}
        We can always rewrite $g(t)$ as $g_{s}(t)+g_{ns}(t)$, where
		\begin{equation*}
			g_{s}(t)=\frac{g(t)+g(t)^{\top}}{2}, \hspace{2em}g_{ns}(t)=\frac{g(t)-g(t)^{\top}}{2}.
		\end{equation*}
		Since $S^{\top}g_{ns}(t)S=0$, the non-symmetric part contributes nothing to quadratic form $S^{\top}g(t)S$. Hence it's quite natural to assume that $g(t)$ is symmetric. Combining equation (\ref{2.17}) and (\ref{2.15}), and with variable separation, we have
		\begin{subequations}
			\begin{align}
			g'(t)=&\frac{1}{1-\gamma}\operatorname{diag}(\alpha)g(t)+\frac{1}{1-\gamma}g(t)\operatorname{diag}(\alpha)-2g(t)\sigma\sigma^{\top}g(t) \notag \\
             & -\frac{\gamma \operatorname{diag}(\alpha)(\sigma\sigma^{\top})^{-1}\operatorname{diag}(\alpha)}{2(1-\gamma)^{2}} + \frac{1}{1-\gamma}\varrho, 
                \label{2.18a}\\
                f'(t) = &\frac{1}{1-\gamma} \operatorname{diag}(\alpha) f(t) - 2 g(t) \sigma \sigma^{\top} f(t) 
				- 2 g(t) \operatorname{diag}(\alpha) w  - \frac{2\gamma}{1-\gamma} g(t) 
				\bigl[ \operatorname{diag}(\alpha) \mu - r e \bigr] \notag\\
				&+ \frac{\gamma}{(1-\gamma)^{2}} \operatorname{diag}(\alpha) 
				(\sigma \sigma^{\top})^{-1} 
				\bigl[ \operatorname{diag}(\alpha) \mu - r e \bigr]+ \frac{1}{1-\gamma} \rho , \label{2.18b}\\
				f_{0}'(t)=&-w^{\top}\operatorname{diag}(\alpha)f(t)-\frac{\gamma}{1-\gamma}\left[\mu^{\top}\operatorname{diag}(\alpha)-re^{\top}\right]f(t)-\frac{1}{2}f^{\top}(t)\sigma\sigma^{\top}f(t) -\mathrm{tr}\left\{g(t)\sigma\sigma^{\top}\right\} \notag \\
               &-\frac{\gamma\left[\mu^{\top}\operatorname{diag}(\alpha)-re^{\top}\right](\sigma\sigma^{\top})^{-1}\left[\operatorname{diag}(\alpha)\mu-re\right]}{2(1-\gamma)^{2}}-\frac{r\gamma-\rho_{0}}{1-\gamma},
                \label{2.18c}
			\end{align}
		\end{subequations}
        with terminal conditions
        \begin{equation*}
			g(T)=0_{n\times n},\hspace{2em}f(T)=0_{n\times 1},\hspace{2em}f_{0}(T)=0.
		\end{equation*}
		This reduces the problem to solving the ODE system (\ref{2.18a})-(\ref{2.18c}).  
		\begin{lemma} \label{l3.3}
			Under Assumptions \ref{a2.1}, \ref{a2.2} and \ref{a2.4}, $g(t),f(t),f_{0}(t)$ are uniformly bounded in Frobenius norm on $[0,T]$.
		\end{lemma}
			It is straightforward to verify that Lemma \ref{l3.3} follows naturally. Upon substituting $t = T - t$, the function $g(t)$ is characterized by a finite, quadratically decreasing term alongside a linearly increasing term. Applying Gronwall's lemma \cite{c680dbbe-119c-31e0-a97a-d790b679674f}, we conclude that $g(t)$ remains uniformly bounded in Frobenius norm over the compact interval $[0, T]$. Similarly, we know that the functions $f(t)$ and $f_0(t)$ are likewise uniformly bounded in Frobenius norm on $[0, T]$.
        \begin{remark} \label{rem:relation}
         Once $g,f$ and $f_0$ are determined, $\varphi_1$ and $\varphi$ can be obtained subsequently. Based on Remark \ref{rr3.1}, this ultimately allows us to derive the optimal policy. 
        \end{remark}
		
	\section{Numerical Methods for Solving ODE System }
        \label{sec4}
		In this section, we focus on solving the ODE system (\ref{2.18a})-(\ref{2.18c}). We present hybrid numerical approaches that integrate exponential Rosenbrock-type methods \cite{doi:10.1137/080717717,LI2021113360} with Runge-Kutta methods to numerically solve (\ref{2.18a})-(\ref{2.18c}). Moreover, the convergence theorem is established.
		\par We observe that (\ref{2.18a}) conforms to the class of symmetric matrix Riccati equations discussed in \cite{LI2021113360}. It presents two numerical methods: ExpEuler, which is second-order, and Erow3, which is third-order. These two methods are particularly well-suited for large stiff ODEs like (\ref{2.18a}), outperforming standard Runge-Kutta methods in such scenarios. Therefore, we employ ExpEuler/Erow3 to solve (\ref{2.18a}). Subsequently, we apply Runge-Kutta methods to solve equations (\ref{2.18b}) and (\ref{2.18c}), as they are well-posed linear ODEs. To preserve the desired order of convergence, we employ second-order and third-order Runge-Kutta schemes, respectively. For simplicity, we denote the ExpEuler method combined with the RK2 scheme by ExpEuler-RK2, and the Erow3 method combined with the RK3 scheme by Erow3-RK3.
		\par To facilitate a clearer presentation of our algorithms, we denote right-hand side of (\ref{2.18a}), (\ref{2.18b}), and (\ref{2.18c}) by $G(g(t)),F(g(t),f(t))$, and $F_{0}(g(t),f(t))$, respectively. And for the sake of completeness in the algorithmic framework, we first revisit the numerical schemes of ExpEuler and Erow3, as introduced in \cite{LI2021113360}. Let $h$ be the time discretization step size and $t_k=hk$. For a given point $g_k$ in the state space, (\ref{2.18a}) is rewritten as 
        \begin{equation*}
            g'(t)=\mathcal{S}_{k}(g)+\mathcal{G}_{k}(g),
        \end{equation*}
        where $\mathcal{S}_{k}(g)$ denotes the Frechet derivative of $G(g)$ at $g_{k}$ and $\mathcal{G}_{k}(g)$ denotes the nonlinear remainder $G(g)-\mathcal{S}_{k}(g)$. Furthermore, exponential integrator is defined as follows: 
        \begin{equation*}
            \mathscr{Y}_j(z)=\int_{0}^{1}\exp\{(1-\theta)z\}\frac{\theta^{j-1}}{(j-1)!}\,\mathrm{d}\theta.
        \end{equation*}
        Then denote the second-order and third-order approximations of $g(t_k)$ by $g_{k}^{(2)}$ and $g_{k}^{(3)}$, respectively. Li et al. \cite{LI2021113360} provide following second-order ExpEuler scheme
		\begin{equation*}
        g_{k+1}^{(2)}=g_{k}^{(2)}+h\mathscr{Y}_1(h\mathcal{S}_{k})(G(g_{k}^{(2)})),
		\end{equation*}
		and third-order Erow3 scheme
		\begin{align*}
			\hat{g}_{k+1}&=g_{k}^{(3)}+h\mathscr{Y}_1(h\mathcal{S}_{k})(G(g_{k}^{(3)})), \hspace{2em} \\
		     g_{k+1}^{(3)}&=\hat{g}_{k+1}+h\mathscr{Y}_1(h\mathcal{S}_{k})(G(g_{k}^{(3)}))+2h\mathscr{Y}_3(h\mathcal{S}_{k})\left(\mathcal{G}_{k}(\hat{g}_{k+1})-\mathcal{G}_{k}(g_{k}^{(3)})\right). 
		\end{align*}
		 \par Now we establish numerical schemes for $f(t),f_{0}(t)$. For $p=2,3$, denote the $p$th-order approximations of $f(t)$ and $f_{0}(t)$ by $f_{k}^{(p)}$ and $f_{0,k}^{(p)}$, respectively.
		The scheme of ExpEuler-RK2 and Erow3-RK3 for solving $f(t)$ takes the form of
		\begin{align}			f_{k+1}^{(p)}&=f_{k}^{(p)}+h\sum_{i=1}^{p}d_{i}^{(p)}y_{k,i}^{(p)}, \label{y0} \\
        y_{k,i}^{(p)}&=F\left(g_{k+l_{i}^{(p)}}^{(p)},f_{k}^{(p)}+h\sum_{j=1}^{i-1}m_{ij}^{(p)}y_{k,j}^{(p)}\right),\hspace{2em}p=2,3.
        \label{y1}
		\end{align}
		where $d_{i}^{(p)},l_{i}^{(p)},m_{ij}^{(p)}$ represents the coefficients for Runge-Kutta schemes with order $p$. We adopt the coefficients of the midpoint method and trapezoidal rule for the cases $p=2$ and $p=3$, respectively (see \cite{BUTCHER1996247}). Similarly, the ExpEuler-RK2 and Erow3-RK3 scheme for solving $f_{0}(t)$ is
		\begin{align*}			f_{0,k+1}^{(p)}&=f_{0,k}^{(p)}+h\sum_{i=1}^{p}d_{i}^{(p)}z_{k,i}^{(p)}, \\	z_{k,i}^{(p)}&=F_{0}\left(g_{k+l_{i}^{(p)}}^{(p)},f_{k+l_{i}^{(p)}}^{(p)}\right),\hspace{2em}p=2,3.
		\end{align*}
		 \par To establish the convergence theorem, we give following lemma first.
        \begin{lemma}
            Under Assumptions \ref{a2.1}, \ref{a2.2} and \ref{a2.4}, the function $F(g, f)$ is Lipschitz continuous with respect to $g$ and $f$, with corresponding Lipschitz constants $L^g$ and $L^f$. Moreover, the function $F_0(g, f)$ is Lipschitz continuous in $g$ and $f$, with corresponding Lipschitz constants $L^{g}_0$ and $L_{0}^f$.
            \label{lip}
        \end{lemma}
        \begin{proof} 
        We present the proof for $F(g, f)$ only, as the argument for $F_0(g, f)$ proceeds in a similar manner. Rewrite $F(g,f)$ as the form of
			\begin{equation*}
F(t,g,f)=Af+gBf+gC+D,
			\end{equation*}
			with 
			\begin{align*}
			    &A=\frac{\operatorname{diag}(\alpha)}{1-\gamma}, \quad B=-2\sigma\sigma^{\top},\quad C=-2\operatorname{diag}(\alpha)w-2\gamma\frac{\operatorname{diag}(\alpha)\mu-re}{1-\gamma}, \\
			&D= \frac{\gamma \operatorname{diag}(\alpha)(\sigma\sigma^{\top})^{-1}\left[\operatorname{diag}(\alpha)\mu-re\right]}{(1-\gamma)^{2}}+ \frac{1}{1-\gamma} \rho.
			\end{align*}
            By the boundedness of functions $g(t)$ and $f(t)$, we have
            \begin{align*}
            \|F(\hat{g},f)-F(g,f)\|=&\left\|(\hat{g}-g)(B_1f+C_1)\right\|\leq L^g\|\hat{g}-g\|, \\
            \|F(g,\hat{f})-F(g,f)\|=&\left\|(A_1+gB_1)(\hat{f}-f)\right\|\leq L^f\|\hat{f}-f\|,
            \end{align*}
            where $L^g$ and $L^f$ are Lipschitz constants.
        \end{proof}
		\begin{theorem}
          Suppose Assumptions \ref{a2.1}, \ref{a2.2}, and \ref{a2.4} hold, and that the time discretization step size $h$ satisfies
          \begin{equation}
              \frac{1}{2}K(K-1) h^{p+1}\leq C_{\mathcal{H}},\hspace{2em} h < \frac{C_p}{L^f},\hspace{2em} p = 2,3,  
              \label{th441}
          \end{equation}
          where $K$ is the number of time discretization steps, $C_{\mathcal{H}}$ is a sufficiently small constant and $C_p > 0$ is a method-dependent constant that may depend on the order $p$. Then ExpEuler-RK2 and Erow3-RK3 is globally second-order and third-order convergent.
            \label{t4.1}
		\end{theorem}
		\begin{proof} 							
			Under the first assumption of (\ref{th441}), the ExpEuler and Erow3 methods achieve global convergence orders of two and three, respectively \cite{doi:10.1137/080717717}. Therefore, for $p=2,3$, we have 
            \begin{equation*}
                \|g(t_{k})-g_{k}^{(p)}\| \leq ch^{p}.
            \end{equation*}
            Let $\hat{y}_{k,i}^{(p)}$ denote the value of (\ref{y1}) by substituting $g_{k+l_{i}^{(p)}}^{(p)}$ with $g(t_{k+l_{i}^{(p)}})$. Correspondingly, let $\hat{f}_i^{(p)}$ denote the numerical solution computed from (\ref{y0}) substituting $y_{k,i}^{(p)}$ with $\hat{y}_{k,i}^{(p)}$. By Lemma \ref{lip}, we know that function $F$ is Lipschitz continuous with respect to $f$. Additionally, by the stability theorem for Runge-Kutta methods in \cite[Chapter 10]{GriffithsHigham2010}, we have that 
            \begin{equation*}
                \|f(t_{k})-\hat{f}_{k}^{(p)}\| \leq ch^{p},
            \end{equation*}
            under assumptions of Lemma \ref{l3.3} and the second assumption of (\ref{th441}). Then for $p=2,3$, we obtain error bound
			\begin{align}
				\|f(t_{k})-f_{k}^{(p)}\| 
                =& \|f(t_{k})-\hat{f}_{k}^{(p)}+\hat{f}_{k}^{(p)}-f_{k}^{(p)}\|  \notag \\
				\leq& \|f(t_{k})-\hat{f}_{k}^{(p)}\|+\|\hat{f}_{k}^{(p)}-f_{k}^{(p)}\| \notag \\
				 \leq& ch^{p} + \|\hat{f}_{k}^{(p)}-f_{k}^{(p)}\|.
                \label{error1}
			\end{align}
            where $c$ is some positive constant determined by Lipschitz constant $L^f$. Throughout the paper, this $c$ may change from one instance to another, depending on the specific terms involved. Denote $\|\hat{f}_{k}^{(p)}-f_{k}^{(p)}\|$ by $e_k^{(p)}$, we have
            \begin{align}
                e_{k+1}^{(p)} \leq&e_{k}^{(p)}+ h \sum_{i=1}^{p}d_i^{(p)}\|\hat{y}_{k,i}^{(p)}-y_{k,i}^{(p)}\| \notag \\
                =&e_{k}^{(p)}+ h \sum_{i=1}^{p}d_i^{(p)}\left\|F\left(g(t_{k+l_{i}^{(p)}}),\hat{f}_{k}^{(p)}+h\sum_{j=1}^{i-1}m_{ij}^{(p)}\hat{y}_{k,j}^{(p)}\right) \right. \notag\\
                 &\hspace{8em}\left.-F\left(g_{k+l_{i}^{(p)}}^{(p)},f_{k}^{(p)}+h\sum_{j=1}^{i-1}m_{ij}^{(p)}y_{k,j}^{(p)}\right)\right\| \notag \\  
                \leq& e_{k}^{(p)}+h \sum_{i=1}^{p}d_i^{(p)}\left(cL^{g}h^{p}+L^{f}e_{k}^{(p)}+hL_f\sum_{j=1}^{i-1}m_{ij}^{(p)}\|\hat{y}_{k,j}^{(p)}-y_{k,j}^{(p)}\|\right) \notag \\
                \leq& e_{k}^{(p)}+ch^{p+1}+che_{k}^{(p)}+ch^2\left(\sum_{i=1}^{p} \sum_{j=1}^{i-1}\left\|\hat{y}_{k,j}^{(p)}-y_{k,j}^{(p)}\right\|\right) \notag \\
                \leq & ce_{k}^{(p)}\sum_{i=0}^{p}h^i+ch^p\sum_{i=1}^{p}h^i \notag \\
                \leq &(1+ch^2)e_{k}^{(p)}+ch^{p+1}. \label{rec}
            \end{align}
          Note that the transition from the third-to-last line to the penultimate line in estimate (\ref{rec}) follows from the application of a recursive estimation approach. For any $1 \leq k \leq K$, by Gronwall's Lemma, we obtain
            \begin{align}
                e_{k}^{(p)} \leq& ch^{p-1}[(1+ch^2)^k-1] \notag\\
                \leq & ch^{p-1}\left(\exp\{ckh^2\}-1\right) \notag \\
                \leq & ch^{p-1} \left(\exp\{cTh\}-1\right) \notag \\
                \leq &ch^{p-1} \left(cTh+O(h^2)\right)\notag \\
                \leq& ch^{p}.
                \label{error2}
            \end{align}
            Combining estimate (\ref{error1}) and (\ref{error2}), the error bound 
            \begin{equation*}
                \|f(t_{k})-f_{k}^{(p)}\| \leq  ch^{p}.
            \end{equation*}
            We observe that the function $F_0(g,f)$ is independent of $f_0$. Similarly, we obtain error bound
			\begin{equation*}
				\|f_{0}(t_{k})-f_{0,k}^{(p)}\| \leq ch^{p}.
			\end{equation*} 
		\end{proof}
\begin{remark} \label{sc}
It is worth noting that when $\sigma\sigma^{\top}$ is diagonal and $\varrho=0_{n \times n}$, the structure of $H(S)$ admits matrix function $g(t)$ to be diagonal as well, (\ref{2.18a}) can be decoupled into $n$ independent $1$-dimensional Riccati equations \cite{Benth01072005}. Consequently, closed-form solutions become tractable. We briefly examine this special case and derive the corresponding closed-form solutions of the ODE system, which can serve as a benchmark for validating the accuracy of our numerical methods in this particular scenario.
Let $(\sigma\sigma^{\top})_{ii}=q_{i}$, and $g(t)=\operatorname{diag}(g_{11}(t),...,g_{nn}(t))$. Substituting $g(t)$ back into (\ref{2.18a})-(\ref{2.18c}) and after some calculations, we have
		\begin{subequations}
			\begin{align}
				g_{ii}(t)=&\frac{\gamma\alpha_{i}}{2(1-\gamma)q_{i}}\frac{\mathrm{sinh}\left[\alpha_{i}(T-t)/\sqrt{1-\gamma}\right]}{\mathrm{sinh}\left[\alpha_{i}(T-t)/\sqrt{1-\gamma}\right]+\sqrt{1-\gamma}\mathrm{cosh}\left[\alpha_{i}(T-t)/\sqrt{1-\gamma}\right]},
                \label{3.1a}\\
				f_{i}(t)=&\int_{t}^{T}\zeta_{i}(u)\exp\left\{\int_{t}^{u}\kappa_{i}(s)\,\mathrm{d}s\right\}\,\mathrm{d}u,
                \label{3.1b}\\
				f_{0}(t)=&\int_{t}^{T}\left\{\sum_{i=1}^{n}\left[w_{i}\alpha_{i}+\frac{\gamma}{1-\gamma}(\mu_{i}\alpha_{i}-r)\right]f_{i}(u)+\frac{1}{2}\sum_{i=1}^{n}q_{i}f_{i}^{2}(u)+\sum_{i=1}^{n}q_{i}g_{ii}(u)\right\}\,\mathrm{d}u\notag \\
				&+\left[\sum_{i=1}^{n}\frac{\gamma(\mu_{i}\alpha_{i}-r)^{2}}{2(1-\gamma)^{2}q_{i}}+\frac{-\rho_{0}+r\gamma}{1-\gamma}\right](T-t),
                \label{3.1c}
			\end{align}
		\end{subequations}
		where 
        \begin{align*}
            \kappa_{i}(t)&=2q_{i}g_{ii}(t)-\alpha_{i}/(1-\gamma),\\
            \zeta_{i}(t)&=2[\alpha_{i}w_{i}+\gamma(\mu_{i}\alpha_{i}-r)/(1-\gamma)]g_{ii}(t)-\gamma\alpha_{i}(\mu_{i}\alpha_{i}-r)/[(1-\gamma)^{2}q_{i}]-\rho_{i}(1-\gamma).
        \end{align*}
\end{remark}
        By combining equations (\ref{2.8}), (\ref{phijia}), (\ref{3phi}), and (\ref{2.15}), when the functions $g$, $f$, and $f_0$ are obtained either analytically or numerically, the solution to the HJB equation (\ref{2.5})-(\ref{2.51}) admits the following explicit form:
		\begin{align}
				v(t,x,S,\psi)=&\gamma^{-1}\psi^{1-\gamma}x^{\gamma}[\varphi_1(t,S)+\varphi_2(t,S)]^{1-\gamma}  \notag \\
                =&\gamma^{-1}\psi^{1-\gamma}x^{\gamma}\Big[\exp\left\{f_{0}(t)+f^{\top}(t)S+S^{\top}g(t)S\right\} \notag \\
                &\hspace{5em}+\int_{t}^{T}\exp\left\{f_{0}(u)+f^{\top}(u)S+S^{\top}g(u)S\right\}\,\mathrm{d}u\Big]^{1-\gamma}.
            \label{4.1}
		\end{align}
		 However, (\ref{4.1}) remains a candidate solution for the value function of the original optimal control problem. Therefore, a verification theorem is essential to validate the optimality of (\ref{4.1}).
	\section{Verification theorem}
        \label{sec5}
		In this section, we establish a rigorous verification theorem that provides sufficient conditions for the candidate solutions to be optimal. Additionally, we prove that verification theorem is valid for numerical solutions as well.
		\par To facilitate the proof of the verification theorem, we first establish the following lemmas.
        \begin{lemma}(\textbf{cf.}\,\cite{1130000794501893504})
        Let $W \sim N(\eta, \Sigma)$ be a $n$-dimensional random variable and given $R\in\mathbb{S}^{n} $, then the moment-generating function of $Q=W^T R W$ is defined as 
\begin{equation*}
\begin{aligned}
M_Q(s) =& \mathbb{E}[\exp\{sQ\}] \\
=&
\begin{cases}
    |I - 2sR\Sigma|^{-\frac{1}{2}} \exp\left\{-\frac{1}{2}\eta^\top[I - (I - 2sR\Sigma)^{-1}]\Sigma^{-1}\eta\right\}, & s \in (s_{\min}, s_{\max}), \\
    +\infty, & \text{otherwise}.
\end{cases}
\end{aligned}
\end{equation*}
Here
\begin{equation*}
\begin{aligned}
    \begin{aligned}
    s_{\min} = \begin{cases} \dfrac{1}{2\lambda_{\min}(R \Sigma )}, &  \lambda_{\min}(R\Sigma ) < 0 ,\\ -\infty, &  \lambda_{\min}(R \Sigma ) \ge 0, \end{cases}
     \end{aligned}
     \quad
    s_{\max} = \begin{cases} \dfrac{1}{2\lambda_{\max}(R \Sigma )}, &  \lambda_{\max}(R \Sigma ) > 0, \\ +\infty, &  \lambda_{\max}(R \Sigma ) \leq 0, \end{cases} \\
    \end{aligned}
\end{equation*}
where $\lambda_{\min}(R \Sigma )$ and $\lambda_{\max}(R \Sigma )$ denote the smallest and largest eigenvalue of matrix $R \Sigma $.
            \label{l5.3}
		\end{lemma}
        \begin{lemma} \label{l5.2}
      (\textbf{cf.}\,\cite{Vatiwutipong2019}) For any terminal time $\tau\in[t,T]$, $S(\tau)$ has an $n$-dimensional normal distribution with mean vector 
      \begin{equation*}
          \eta(\tau)=\exp\{-(\tau-t)\operatorname{diag}(\alpha)\}S(t)+\left(I-\exp\{-(\tau-t)\operatorname{diag}(\alpha)\}\right)w,
      \end{equation*}
      and covariance matrix
              \begin{equation*}
\left(\Sigma(\tau)\right)_{ij}=\frac{(\sigma_{i}\sigma_{j}^{\top})(1-\exp\{-(\tau-t)(\alpha_{i}+\alpha_{j})\})}{\alpha_{i}+\alpha_{j}}.
	\end{equation*}
      Here $\Sigma(\tau) \in\mathbb{S}_{++}^{n} $, where $\mathbb{S}_{++}^{n} $ denotes the set of all symmetric positive definite $n \times n $ matrices.
		\end{lemma}
        \begin{lemma}(\textbf{cf.}\,\cite{doi:10.1137/0325086}) 
             If condition 
			\begin{equation}
				\mathbb{E}\left[\exp\left\{\max\left\{\frac{1}{2},\frac{\gamma^{2}}{2(1-\gamma)^{2}}\right\}\int_{0}^{T}\|\sigma^{-1}(\operatorname{diag}(\alpha)(\mu-S(t))-re)\|^{2}\,\mathrm{d}t\right\}\right]<+\infty
				\label{4.2}
			\end{equation}
			holds for any $t\in[0,T]$,
             then Novikov condition (\ref{2.11}) is guaranteed and we can also define a new probability measure $\tilde{P}$ where condition (d) in Definition \ref{d2.5} is equivalent to 
			\begin{equation*}
				\tilde{\mathbb{E}}\left[\int_{0}^{T}C(t)\exp\{-rt\}\,\mathrm{d}t\right] \leq x,\quad
				\tilde{\mathbb{E}}\left[X(T)\exp\{-rT\}\right] \leq x.
			\end{equation*}
            \label{l5.4}
		\end{lemma}
		\begin{lemma} \label{l5.1}
			(\textbf{cf.}\,\cite{Benth01072005}) Let function $v(t,x,S,\psi)$ be a classical solution of HJB equation (\ref{2.5})-(\ref{2.51}) and the family $\left\{v\left(\tau,X^{(\pi^{*},C^{*})}(\tau),S(\tau),\psi(\tau)\right)\right\}_{\tau \in [0,T]}$ is uniformly integrable. Then we have 
            \begin{equation*}
                v(t,x,S,\psi)=V(t,x,S,\psi).
            \end{equation*}
		\end{lemma}
            The $1$-dimensional case of Lemma \ref{l5.1} is given in \cite[Lemmas 4.1-4.2]{Benth01072005}, and the extension to the $n$-dimensional case follows naturally. We omit the details here. 
        
		\begin{theorem}
			Under Assumptions \ref{a2.1}, \ref{a2.2}, \ref{a2.4}, for any $\tau \in [0,T]$, suppose the following conditions hold:
			\begin{align}
				\lambda_{\max} (g(\tau) \Sigma(\tau))<& \frac{1}{4(1-\gamma)} , \label{4.3} \\
                -\lambda_{\min} (\varrho \Sigma(\tau)) <& \frac{1}{8(1-\gamma)} , \label{4.4}                \\
				\lambda_{\max} (\Gamma \Sigma(\tau) )<& \frac{1}{8\gamma} ,\label{4.5}\\
				\lambda_{\max} (\Pi(\tau) \Sigma(\tau))<&\frac{1}{256\gamma^2},
                 \label{4.6}
			  \end{align}
			where 
            \begin{align*}
            \Gamma &=\operatorname{diag}(\alpha)(\sigma \sigma^{\top})^{-1}\operatorname{diag}(\alpha), 
            \\
            \Pi(\tau)&= 4g(\tau) \sigma \sigma^T g(\tau)+\frac{1}{1-\gamma} \operatorname{diag}(\alpha) \sigma \sigma^T \operatorname{diag}(\alpha),
             \end{align*}
             then the candidate solution $v$ given in (\ref{4.1}) is indeed the value function (\ref{o3}). Furthermore, suppose that for any $0 \leq t \leq u \leq T$, it also has
			\begin{align}
				\lambda_{\max}(\Gamma \Sigma(t))<&\min\left\{\frac{1}{8},\frac{(1-\gamma)^{2}}{\gamma^{2}}\right\},
                \label{4.7}\\
                \lambda_{\max} (\Pi(t) \Sigma(t))<&\frac{1}{256},
                \label{4.8}\\
				-\lambda_{\min} (g(t) \Sigma(t)) <& \frac{1}{4} ,
                \label{4.9} \\
				\lambda_{\max} ([(g(u)-g(t)]\Sigma(t)) < & \frac{1}{16}, 
                \label{4.10}
                \end{align}
			and the initial wealth $x$ is chosen such that 
			\begin{equation}
				x=\max\left\{\tilde{\mathbb{E}}\left[\int_{0}^{T}C(t)\exp\{-rt\}\,\mathrm{d}t\right],\tilde{\mathbb{E}}\left[X(T)\exp\{-rT\}\right] \right\},
                \label{4.13}
			\end{equation}
			then optimal portfolio-consumption policies are given by equations (\ref{3.221})-(\ref{3.222}).
		\end{theorem}		
		\begin{proof}
			\par From the statement of Lemma \ref{l5.1}, to establish the optimality of the candidate solution $v$, it suffices to prove that $\left\{v\left(\tau,X(\tau)^{(\pi^{*},C^{*})},S(\tau),\psi(\tau)\right)\right\}_{\tau \in [0,T]}$ is uniformly integrable. It is notable that for any given $\epsilon>0$ and stopping time $\tau \in [0,T]$, by equation (\ref{4.1}) and Holder's inequality, we have
			\begin{align}
				&\,\mathbb{E}^{(t,x,S,\psi)}\left[v^{1+\epsilon}(\tau,X^{*}(\tau),S(\tau),\psi(\tau))\right] \notag\\
				=&\,\gamma^{-(1+\epsilon)}\mathbb{E}^{(t,x,S,\psi)}\left[\left(\sum_{i=1}^2\varphi_{i}(\tau,S(\tau))\right)^{(1-\gamma)(1+\epsilon)}\psi(\tau)^{(1-\gamma)(1+\epsilon)}(X^{*}(\tau))^{\gamma(1+\epsilon)}\right] \notag \\
				\leq&\, c\mathbb{E}^{(t,x,S,\psi)}\left[\left(\sum_{i=1}^2\varphi_{i}^{(1-\gamma)(1+\epsilon)}(\tau,S(\tau))\right)\psi(\tau)^{(1-\gamma)(1+\epsilon)}(X^{*}(\tau))^{\gamma(1+\epsilon)}\right] \notag \\
				=&\, c\sum_{i=1}^2 \mathbb{E}^{(t,x,S,\psi)}\left[\varphi_{i}^{(1-\gamma)(1+\epsilon)}(\tau,S(\tau))\psi(\tau)^{(1-\gamma)(1+\epsilon)}(X^{*}(\tau))^{\gamma(1+\epsilon)}\right] \notag \\
                \leq&\, c\sum_{i=1}^2\left\{\mathbb{E}^{(t,x,S,\psi)}\left[\varphi_{i}^{2(1-\gamma)(1+\epsilon)}(\tau,S(\tau))\right]\right\}^{\frac{1}{2}}\left\{\mathbb{E}^{(t,x,S,\psi)}\left[\psi(\tau)^{4(1-\gamma)(1+\epsilon)}\right]\right\}^{\frac{1}{4}} \notag \\
                &\,\times \left\{\mathbb{E}^{(t,x,S,\psi)}\left[(X^{*}(\tau))^{4\gamma(1+\epsilon)}\right]\right\}^{\frac{1}{4}} . \label{fs1}
			\end{align}
			 Therefore, to demonstrate the integrability, we only need to prove that for any $\tau\in[t,T] $, the right-hand side of above equation is bounded.
             \par We first handle the first expectation term in the right-hand side of (\ref{fs1}), by equations (\ref{3phi}), (\ref{2.15}) and Lemma \ref{l3.3}, we obtain
				\begin{align*}					&\mathbb{E}^{(t,x,S,\psi)}\left[\varphi_{1}^{2(1-\gamma)(1+\epsilon)}(\tau,S(\tau))\right]\\
					=&\mathbb{E}^{(t,x,S,\psi)}[\exp\{2(1-\gamma)(1+\epsilon)(f_{0}(\tau)+f(\tau)^{\top}S(\tau)+S(\tau)^{\top}g(\tau)S(\tau))\}]\\
					\leq& c\mathbb{E}^{(t,x,S,\psi)}[\exp\{2(1-\gamma)(1+\epsilon)(f(\tau)^{\top}S(\tau)+S(\tau)^{\top}g(\tau)S(\tau))\}].
				\end{align*}
				 Furthermore, by Young's inequality \cite{Benth01072005}, it has
				\begin{equation*}
					2(1-\gamma)(1+\epsilon)f^{\top}(\tau)S(\tau) \leq c+S(\tau)^{\top}\delta S(\tau),\quad \forall  \tau \in[t,T],
				\end{equation*}
				where $\delta$ denotes an arbitrary positive definite diagonal matrix, and $c$ is determined by the upper bound of $f_0(\tau),f_0(\tau)$ and $\delta$. Hereafter, each occurrence of $\delta$ corresponds to a new application of the matrix form of Young's inequality. Based on the discussion above, we have
				\begin{align}				&\mathbb{E}^{(t,x,S,\psi)}\left[\varphi_{1}^{2(1-\gamma)(1+\epsilon)}(\tau,S(\tau))\right] \notag \\
                \leq& c\mathbb{E}^{(t,x,S,\psi)}[\exp\{S(\tau)^{\top}[2(1-\gamma)(1+\epsilon)g(\tau)+\delta]S(\tau)\}].
                \label{4.14}
				\end{align}
				Let us set 
                \begin{equation*}
                    s=2,\quad R=2(1-\gamma)(1+\epsilon)g(\tau)+\delta.
                \end{equation*}
                By Lemma \ref{l5.3}-\ref{l5.2}, given the arbitrariness of $\epsilon$ and matrix $\delta$, right-hand side of estimate (\ref{4.14}) is bounded when condition (\ref{4.3}) is satisfied. Moreover, by equation (\ref{3phi}), the first expectation term in the right-hand side of (\ref{fs1}) is also bounded under the same condition.
				\par Next we proceed to prove that the second expectation term of the right-hand side of (\ref{fs1}) is bounded. Review equation (\ref{psi}) and by Young's inequality, we obtain
				\begin{align}					&\mathbb{E}^{(t,x,S,\psi)}\left[\psi(\tau)^{4(1-\gamma)(1+\epsilon)}\right] \notag\\
					=&\mathbb{E}^{(t,x,S,\psi)}\left[\exp \left\{-4(1-\gamma)(1+\epsilon)\int_{0}^{\tau}[\rho_{0}+S(u)^{\top}\rho+S(u)^{\top}\varrho S(u)] \,\mathrm{d}u\right\}\right] \notag \\
					\leq& c\int_{0}^{\tau}\mathbb{E}^{(t,x,S,\psi)}\left[\exp \left\{S(u)^{\top}[\delta-4(1-\gamma)(1+\epsilon)\varrho]S(u) \right\}\right] \,\mathrm{d}u.
					\label{4.15}
				\end{align}
				 Similarly, by Lemma \ref{l5.3}-\ref{l5.2} and the arbitrariness of $\epsilon$ and $\delta$, the expectation on the right-hand side of estimate (\ref{4.15}) is bounded if condition (\ref{4.4}) is guaranteed.
			\par Finally, let us prove that the third expectation term of the right-hand side of (\ref{fs1}) is bounded. We denote the wealth when consumption is absent ($C\equiv0$) by $Z(t)$ and the corresponding optimal wealth by $Z^{*}(t)$. Obviously, we have that for any $t\in [0,T],X^{*}(t)\leq Z^{*}(t)$. Thus, we only need to prove $\mathbb{E}^{(t,x,S,\psi)}\left[(Z(\tau))^{4\gamma(1+\epsilon)}\right]<+\infty$. Following the proof of Theorem 4.4 of \cite{Benth01072005}, by equation (\ref{ss}), we can readily derive 
				\begin{align}					   &\mathbb{E}^{(t,x,S,\psi)}\left[(Z(\tau))^{4\gamma(1+\epsilon)}\right] \notag\\ 
				\leq&c\int_{t}^{\tau}\mathbb{E}^{(t,x,S,\psi)}\left[\exp \left\{8\gamma(1+\epsilon)S(u)^{\top}(\delta+\frac{1}{2}\operatorname{diag}(\alpha)(\sigma \sigma^{\top})^{-1}\operatorname{diag}(\alpha))S(u)\right\}\right]\,\mathrm{d}u \notag\\
			    	& \times \int_{t}^{\tau} \mathbb{E}^{(t,x,S,\psi)} \bigg[ \exp \bigg\{ 128 \gamma^2 (1+\epsilon)^2 S(u)^T  [\delta + 4g(t) \sigma \sigma^T g(t)\notag\\
                    & \hspace{9em}+ \frac{1}{1-\gamma} \operatorname{diag}(\alpha) \sigma \sigma^T \operatorname{diag}(\alpha)] S(u)  \bigg\} \bigg]^{\frac{1}{4}}\,\mathrm{d}u.
                    \label{4.16}
				\end{align}
				 By Lemma \ref{l5.3}-\ref{l5.2} and the arbitrariness of $\epsilon$ and matrix $\delta$, right-hand integral of estimate (\ref{4.16}) is bounded by conditions (\ref{4.5})-(\ref{4.6}). Consequently, the boundedness of the third expectation in (\ref{fs1}) is assured by these same conditions. So far, we have proved that the candidate solution $v$ is the value function $V$ if conditions (\ref{4.3})-(\ref{4.6}) are satisfied.
			\par Now let us prove that conditions (\ref{4.7})-(\ref{4.10}) guarantee $(\pi^{*},C^{*})$ to be admissible. We have proved that for $\left\{v\left(\tau,X^{(\pi^{*},C^{*})}(\tau),S(\tau),\psi(\tau)\right)\right\}_{\tau}$ is uniformly integrable under conditions (\ref{4.3})-(\ref{4.6}). Therefore, condition $(c)$ in Definition \ref{d2.5} is also guaranteed. Next we will prove rest of the admissibility conditions. 
            
                Firstly, to ensure the feasibility of the change of measure in Section \ref{sec3.2} and Lemma \ref{l5.4}, it's sufficient to verify that condition (\ref{4.2}) is satisfied. By Young's and Jensen's inequalities, we derive 
               \begin{align*}
                   &\mathbb{E}\left[\exp\left\{\max\left\{\frac{1}{2},\frac{\gamma^{2}}{2(1-\gamma)^{2}}\right\}\int_{0}^{T}\|\sigma^{-1}\operatorname{diag}(\alpha)(\mu-S(u))\|^{2}\,\mathrm{d}u\right\}\right]\\
                    \leq& c\int_{0}^{T} \mathbb{E}\left[\exp\left\{\max\left\{\frac{1}{2},\frac{\gamma^{2}}{2(1-\gamma)^{2}}\right\}S(u)^{\top}\left[\delta+\operatorname{diag}(\alpha)(\sigma\sigma^{\top})^{-1}\operatorname{diag}(\alpha)\right]S(u)\right\}\right]\,\mathrm{d}u.
                \end{align*}
				 Then, by Lemmas \ref{l5.3}-\ref{l5.2} and the arbitrariness of $\delta$, we can conclude that condition (\ref{4.7}) ensures condition (\ref{4.2}). Consequently, with $x$ defined in equation (\ref{4.13}), condition (d) in Definition \ref{d2.5} is fulfilled, according to Lemma \ref{l5.4}.	
                 
			 Then let us verify that $C^{*}$ satisfies condition (a) in Definition \ref{d2.5}. From the analysis in estimate (\ref{4.16}), we conclude that  $\mathbb{E}[X^{*}(t)]^{4}<+\infty$ holds for any $t\in[0,T]$ under conditions (\ref{4.7})-(\ref{4.8}), which helps us to prove the $L^{1}$ integrability of $C^{*}$. By applying equations (\ref{3.222}) and (\ref{2.15}), together with Cauchy's inequality and Young's inequality, we obtain 
                \begin{align}
                &\mathbb{E}\left[\int_{0}^{T}C^{*}(t)\,\mathrm{d}t\right] \notag \\
                     \leq&\int_{0}^{T}\mathbb{E}\left[\varphi^{-1}(t,S(t))|X^{*}(t)|\right]\,\mathrm{d}t \notag\\
					\leq& \int_{0}^{T}\left\{\mathbb{E}\left[\varphi^{-2}(t,S(t))\right]\right\}^{\frac{1}{2}}\left\{\mathbb{E}\left[(X^{*}(t))^{2}\right]\right\}^{\frac{1}{2}}\,\mathrm{d}t \notag\\
					\leq& c\int_{0}^{T}\{\mathbb{E}[\exp\{-2(f_{0}(t)+S(t)^{\top}f(t)+S(t)^{\top}g(t)S(t))\}]\}^{\frac{1}{2}}\,\mathrm{d}t \notag \\
                    \leq& c \int_{0}^{T}\{\mathbb{E}[\exp\{S(t)^{\top}[\delta-2g(t)]S(t)\}]\}^{\frac{1}{2}}\,\mathrm{d}t.
                    \label{123}
				\end{align} 
                 In the same way, by Lemma \ref{l5.3}-\ref{l5.2} and arbitrariness of $\delta$, if condition (\ref{4.9}) is satisfied, the right-hand side of estimate (\ref{123}) is bounded, which directly guarantees the $L^{1}$ integrability of $C^{*}$.
				\par  At last, to establish the $L^{1}$ integrability of $\pi^{*}$, by equations (\ref{3.221}) and (\ref{2.15}), Lemma \ref{l3.3}, Holder's inequality and Jensen's inequality, we obtain 
				\begin{align}
					&\mathbb{E}\left[\int_{0}^{T}\|\pi^{*}(t)\|^{2}\,\mathrm{d}t\right] \notag\\
					\leq& c+2c\int_{0}^{T}\mathbb{E}\left[\varphi^{-2}(t,S(t))\left\|D_{S}\varphi(t,S(t))\right\|^{2}(X^{*}(t))^{2}\right]\,\mathrm{d}t \notag\\
					\leq& c+2c\left\{\int_{0}^{T}\mathbb{E}\left[\varphi^{-4}(t,S(t))\|D_{S}\varphi(t,S(t))\|^{4}\right]\,\mathrm{d}t\right\}^\frac{1}{2} \left\{\int_{0}^{T}\mathbb{E}\left[(X^{*}(t))^{4}\right]\,\mathrm{d}t\right\}^\frac{1}{2} \notag\\
					\leq&  c+c\Bigg\{\int_{0}^{T}\mathbb{E}\bigg[\Big\|\int_{t}^{T}\exp\left\{f_{0}(u)+S(t)^{\top}(f(u)-f(t))+S(t)^{\top}(g(u)-g(t))S(t)\right\}\notag\\
                    &\hspace{7em}
                    \times (f(u)+2g(u)S(t))\,\mathrm{d}u  +f(t)+2g(t)S(t)\Big\|^{4}\bigg]\,\mathrm{d}t\Bigg\}^\frac{1}{2}\notag  \\
					\leq& c+c\left\{\int_{0}^{T}\int_{t}^{T}\left\{\mathbb{E}\left[\exp\left\{S(t)^{\top}[8(g(u)-g(t))+\delta]S(t)\right\}\right]\right\}^{\frac{1}{2}}\,\mathrm{d}u\,\mathrm{d}t\right\}^\frac{1}{2}.
					\label{4.17}
				\end{align}
				 By Lemma \ref{l5.3}-\ref{l5.2}, condition (\ref{4.10}) guarantees the boundedness of the right-hand side of estimate (\ref{4.17}), and therefore, guarantees the integrability of $\pi^{*}$. 
                 \par By Pontryagin's Maximum Principle, we know that $(\pi^{*},C^{*})$ maximizes the Hamiltonian $\mathcal{G}_{xS \psi}$. Therefore, it's the optimal policy. The proof of the theorem is completed. 
		\end{proof}		
		\begin{corollary}
			Under the assumptions of Theorem \ref{t4.1}, let $\hat{v}(t,x,S,\psi)$ denotes the numerical approximation of $v(t,x,S,\psi)$ using ExpEuler-RK2, Errow3-RK3 methods, if the discretization step $h$ is small enough and numerical solution $\hat{g}(t)$ satisfies conditions (\ref{4.3})-(\ref{4.10}), then $\hat{v}(t,x,S,\psi)$ provides a good approximation for the value function $V(t,x,S,\psi)$. Moreover, the optimal policy can be approximated by computing the feedback control given in (\ref{3.221})-(\ref{3.222}) using the numerical solution $\hat{\varphi}(t,S)$.
		\end{corollary}
		\begin{proof}
			Combining Remark \ref{rem:relation} with Theorem \ref{t4.1}, for any given compact set $\mathcal{K}_1 \subset [0,T]\times \mathbb{R}^n$ and any $\varepsilon>0$, there exists a sufficiently small discretization step $h$, such that
			\begin{align*}
				\|g(t)-\hat{g}(t)\| &\leq \varepsilon, \hspace{2em} \forall t \in [0,T],\\
				|\varphi(t,S)-\hat{\varphi}(t,S)| &\leq \varepsilon, \hspace{2em} \forall(t,S) \in \mathcal{K}_1.
			\end{align*}
			When $\hat{g}(t)$ satisfies the verification conditions, it's reasonable to conclude that $g(t)$ also satisfies them, given the continuity of the functions $\lambda_{\max}(\cdot)$ and $\lambda_{\min}(\cdot)$. Hence, our numerical methods offer a reliable approximation of $v$, and therefore, of the value function $V$. 
			Moreover, we note that both $\varphi$ and $\hat{\varphi}$ exhibit smoothness, then we can compute the gradient $D_{S}\hat{\varphi}(t,S)$ to approximate $D_{S}{\varphi}(t,S)$. By formulas (\ref{3.221})-(\ref{3.222}), for any given compact set $\mathcal{K}_2 \subset [0,T] \times \mathbb{R} \times \mathbb{R}^n \times \mathbb{R}_+$ and any $\varepsilon>0$, there also exists a sufficiently small discretization step $h$, such that
			\begin{align*}
				\|\pi^{*}(t,x,S,\psi)-\hat{\pi}^{*}(t,x,S,\psi)\|&\leq \varepsilon, \hspace{2em} \forall(t,x,S,\psi)\in \mathcal{K}_2,\\
				|C^{*}(t,x,S,\psi)-\hat{C}^{*}(t,x,S,\psi)|&\leq \varepsilon, \hspace{2em} \forall(t,x,S,\psi)\in \mathcal{K}_2,
			\end{align*}
			where $(\hat{\pi}^{*},\hat{C}^{*})$ denotes the numerical optimal policy. Hence, it is reasonable to establish that $(\hat{\pi}^{*},\hat{C}^{*})$ is a good approximation of optimal policy. 	
		\end{proof}
        
\section{Numerical Experiments}	\label{sec6}
In this section, we present several numerical experiments to assess the performance of ExpEuler-RK2 and Erow3-RK3.
First, we examine the accuracy of these two methods in solving the ODE system (\ref{2.18a})-(\ref{2.18c}). We then test the convergence order and computational efficiency of ExpEuler-RK2 and Erow3-RK3, verifying the results presented by Theorem \ref{t4.1}. Additionally, we conduct a comparison with a conventional grid-based method to demonstrate the superiority of Erow3-RK3. Finally, we apply Erow3-RK3 to a real-world market case to showcase its practical application.
All experiments are performed under Windows 11 and Matlab R2024b running on a laptop with 13th Gen Intel Core i7-13650HX processor with 2.60 GHz and RAM 16 GB.

\subsection{Accuracy and Efficiency}
\label{sec6.1}
To evaluate ExpEuler-RK2 and Erow3-RK3, we consider the special case in Remark \ref{sc}, which admits closed-form solutions. Let us set time horizon $t\in[0,1]$ and the coefficient settings are as follows:
\begin{align*}
&r=0.5,\quad \gamma=0.5,\quad	 \rho_{0}=0, \quad  \rho=\mathrm{zeros}(n,1), \quad \alpha=0.3 + 0.4 \times \mathrm{rand}(n, 1), \\
  & \mu=5+3\times \mathrm{rand}(n,1) \quad\sigma=\mathrm{orth}(0.01\times \mathrm{rand}(n,n)), \quad \varrho=\mathrm{zeros}(n,n).
\end{align*}

\begin{figure}[h]
	\centering
	\begin{subfigure}[b]{0.455\textwidth} 
		\centering
		\includegraphics[width=\textwidth]{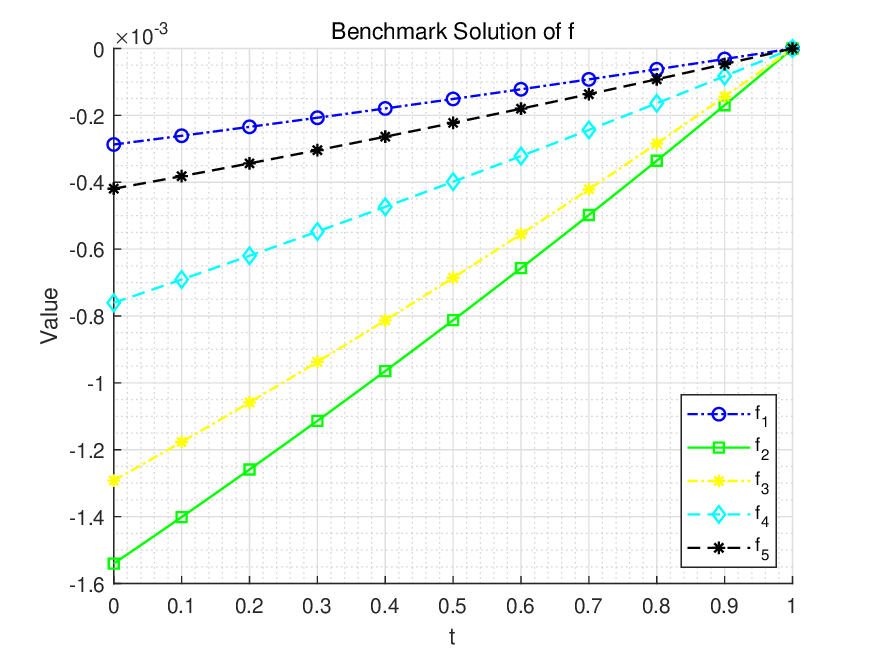} 
            \caption{}
            \label{f2a}
	\end{subfigure}
	\quad
	\begin{subfigure}[b]{0.455\textwidth} 
		\centering
		\includegraphics[width=\textwidth]{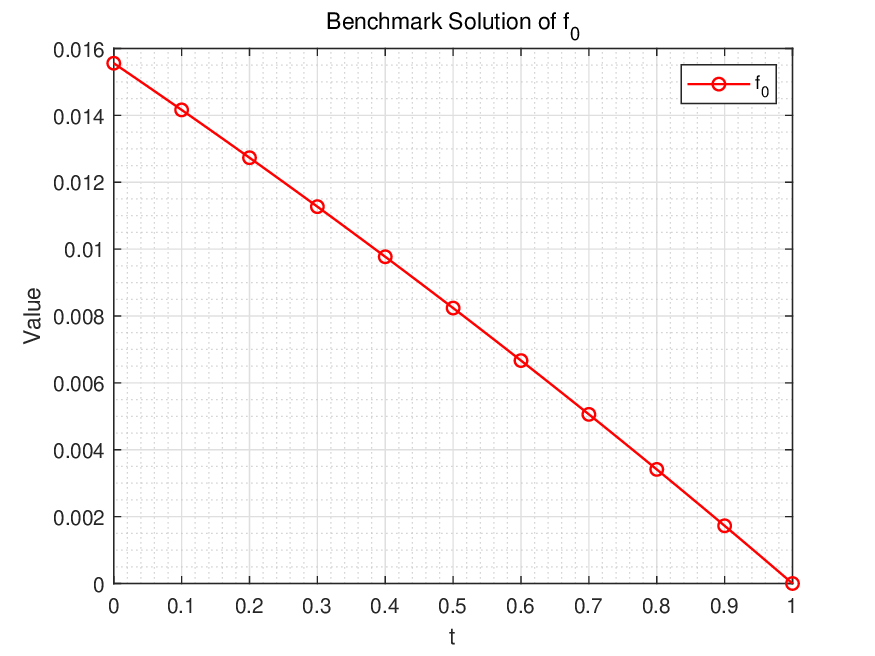} 
            \caption{}
            \label{f2b}
	\end{subfigure}
    \begin{subfigure}[b]{0.455\textwidth} 
		\centering
		\includegraphics[width=\textwidth]{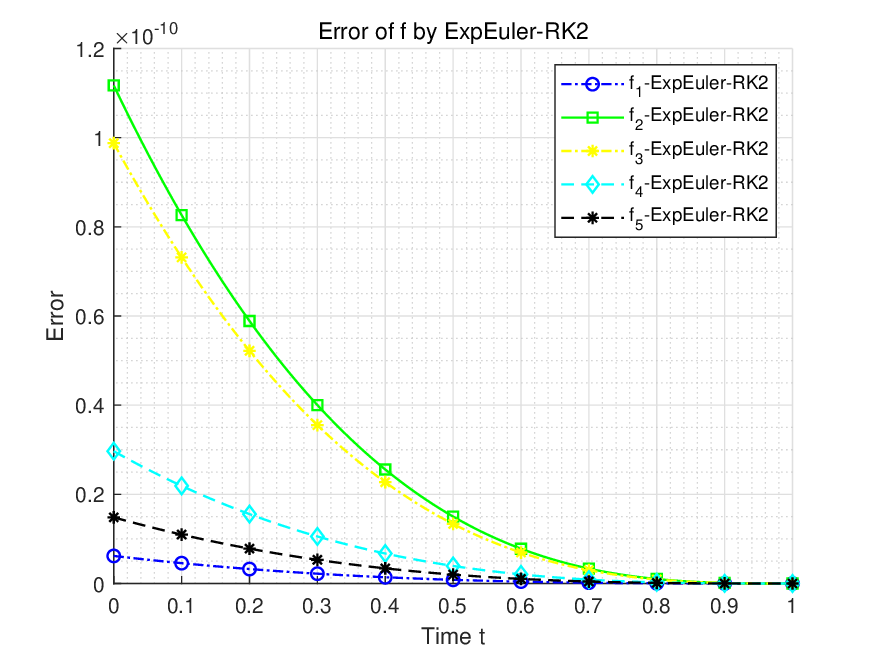} 
            \caption{}
            \label{f2c}
	\end{subfigure}
	\quad
	\begin{subfigure}[b]{0.455\textwidth} 
		\centering	
		\includegraphics[width=\textwidth]{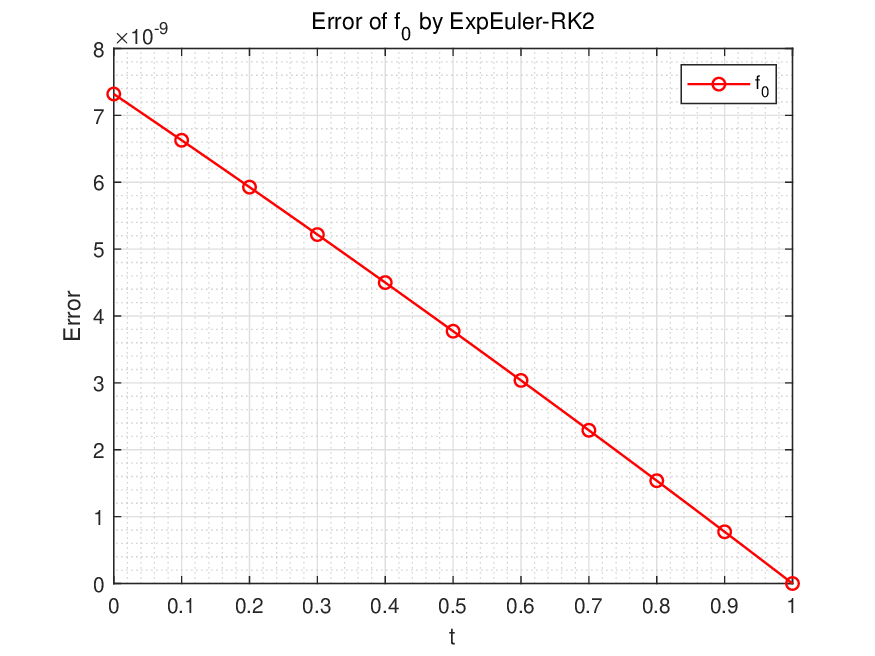} 
            \caption{}
            \label{f2d}
\end{subfigure}
\begin{subfigure}[b]{0.455\textwidth} 
		\centering
		\includegraphics[width=\textwidth]{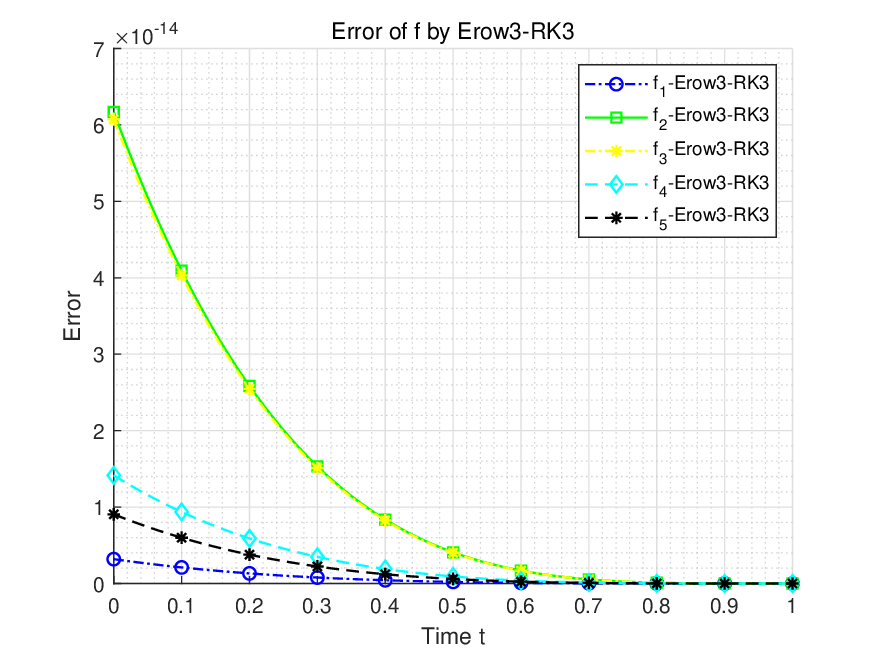}
		\caption{}
        \label{f2e}
	\end{subfigure}
	\quad
	\begin{subfigure}[b]{0.455\textwidth} 
		\centering	
		\includegraphics[width=\textwidth]{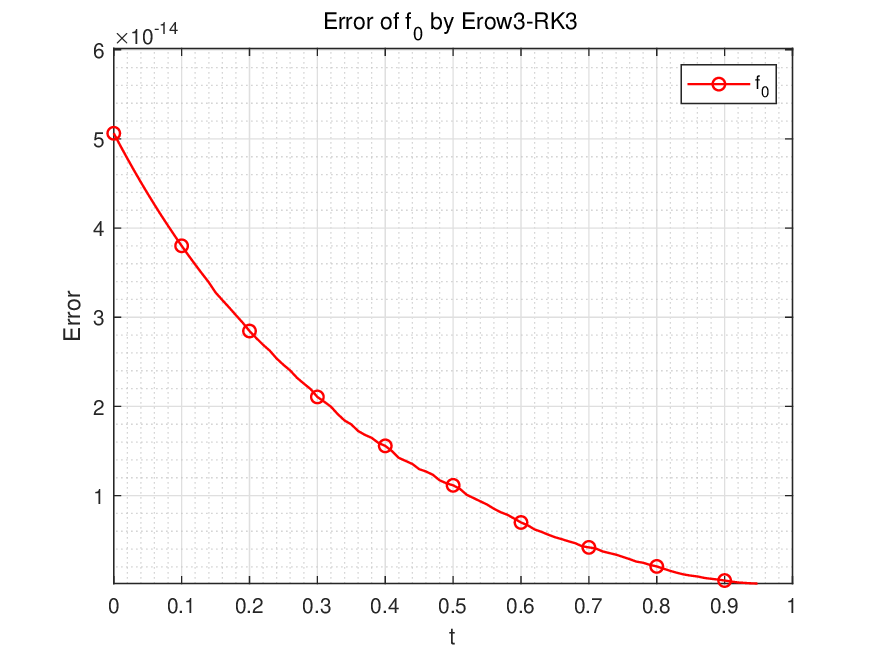}
            \caption{}
            \label{f2f}
	\end{subfigure}
\caption{Benchmark solutions and Numerical Errors}
\label{f2}
\end{figure}

Closed-form solutions of (\ref{2.18a})-(\ref{2.18c}) are given in equations (\ref{3.1a})-(\ref{3.1c}), we divide the interval $[0,1]$ into 100 uniform subintervals (i.e., with a step size of $h = 0.01$), and apply the Composite Simpson’s numerical integration rule to compute benchmark solutions. These benchmarks are then used to examine the accuracy and convergence order of our methods. Since ExpEuler and Erow3 are already proved to be great solvers for matrix Riccati equations like (\ref{2.18a}) in \cite{LI2021113360}, we will only evaluate the performance of our methods on $f(t)$ and $f_{0}(t)$. We first evaluate the accuracy of ExpEuler-RK2 and Erow3-RK3 in the case of $n=5$.

In Figure \ref{f2}, panel (\ref{f2a})-(\ref{f2b}) present the benchmark solution for each entry of $f(t)$ and $f_{0}(t)$. Panels (\ref{f2c})-(\ref{f2d}) and (\ref{f2e})-(\ref{f2f}) illustrate the numerical errors obtained using ExpEuler-RK2, Erow3-RK3 for solving $f(t)$ and $f_0(t)$, respectively. The figure demonstrates that both schemes solve the ODE system with high accuracy. And Errow3-RK3 delivers markedly superior precision compared with ExpEuler-RK2. Moreover, we observe that the numerical error decreases as time progresses, which is attributed to the terminal nature of the problem.

In addition, we test convergence order of these two algorithms in terms of solving the ODE system. We set $n=10$ and select different step sizes: $h=2^{-k},k=3,\dots,7$. The relative errors are measured in the Frobenius norm for matrix and vector functions. In Figure \ref{f3}, panels (\ref{test1}) and (\ref{test3}) show the logarithmic values of the numerical errors for ExpEuler-RK2 and Erow3-RK3 against the degree of freedom, which shows that convergence orders are precisely as established in Theorem \ref{t4.1}. And panels (\ref{test2}) and (\ref{test4}) illustrate the error against computation time (seconds), which shows that Erow3-RK3 outperforms ExpEuler-RK2 in terms of efficiency. Therefore, we use Erow3-RK3 for the remaining experiments.

\begin{figure}[htbp]
	\centering
	\begin{subfigure}[b]{0.46\textwidth} 
	\centering
	\includegraphics[width=\textwidth]{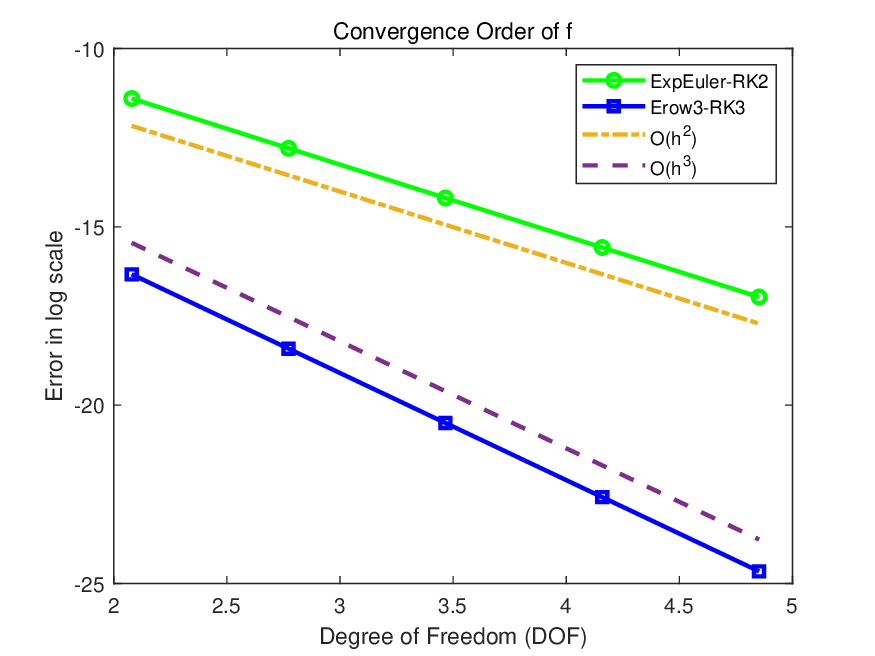} 
	\caption{}
    \label{test1}
\end{subfigure}
\quad
\begin{subfigure}[b]{0.46\textwidth} 
	\centering 
	\includegraphics[width=\textwidth]{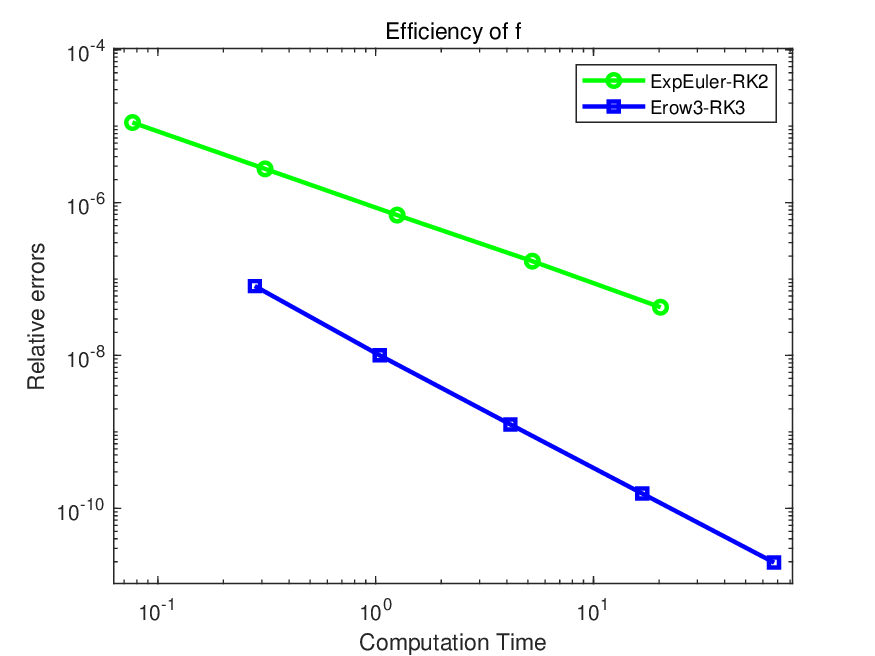} 
	\caption{}
    \label{test2}
\end{subfigure}
\begin{subfigure}[b]{0.46\textwidth} 
\centering
\includegraphics[width=\textwidth]{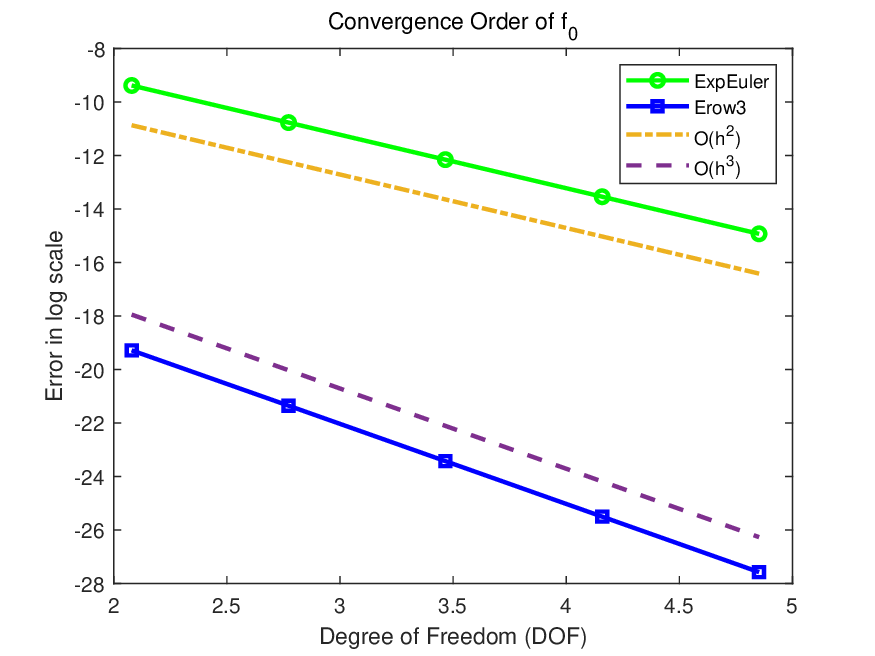} 
\caption{}
\label{test3}
\end{subfigure}
\quad
\begin{subfigure}[b]{0.46\textwidth} 
\centering
\includegraphics[width=\textwidth]{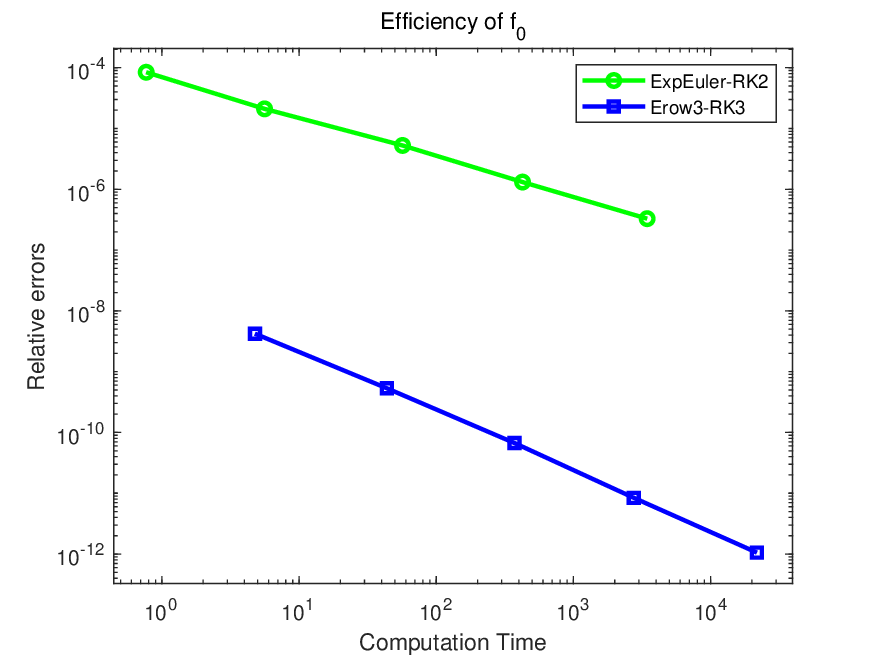} 
\caption{}
\label{test4}
\end{subfigure}
\caption{Order and Efficiency Test}
\label{f3}
\end{figure}

\subsection{Advantage of Erow3-RK3}
To compare our method with grid-based approach, we focus on the solvers of PDE (\ref{2.9})-(\ref{2.10}) in $1$-dimensional case. The coefficients and intervals of the problem are selected as 
\begin{align*}
    &r=0.01,\quad \gamma=0.5,\quad \rho_{0}=0.01,\quad \rho=0,\quad \alpha=0.005,\\
    & \mu=3, \quad \sigma=1, \quad t\in[0,1], \quad S\in[-10,10].
\end{align*}
 It is notable that the analytical approach is feasible in $1$-dimensional case, therefore we compute the benchmark solution using the numerical integration method in Section \ref{sec6.1}. Then we employ Erow3-RK3 to solve $\varphi(t,S)$ with $25$ time discretization steps. Additionally, we directly solve PDE (\ref{2.9})-(\ref{2.10}) using the integration-difference method, which employs the Finite Difference Method (FDM) with exact boundary conditions obtained through numerical integration. This approach is less computationally demanding than the benchmark solution and more accurate than the FDM with truncated boundary conditions. We discretize both time and space with 100 uniform steps for the integration-difference method and compare its numerical error with Erow3-RK3.

\begin{figure}[h]
    \centering
    \begin{subfigure}[b]{0.45\textwidth}
        \centering
        \includegraphics[width=\textwidth]{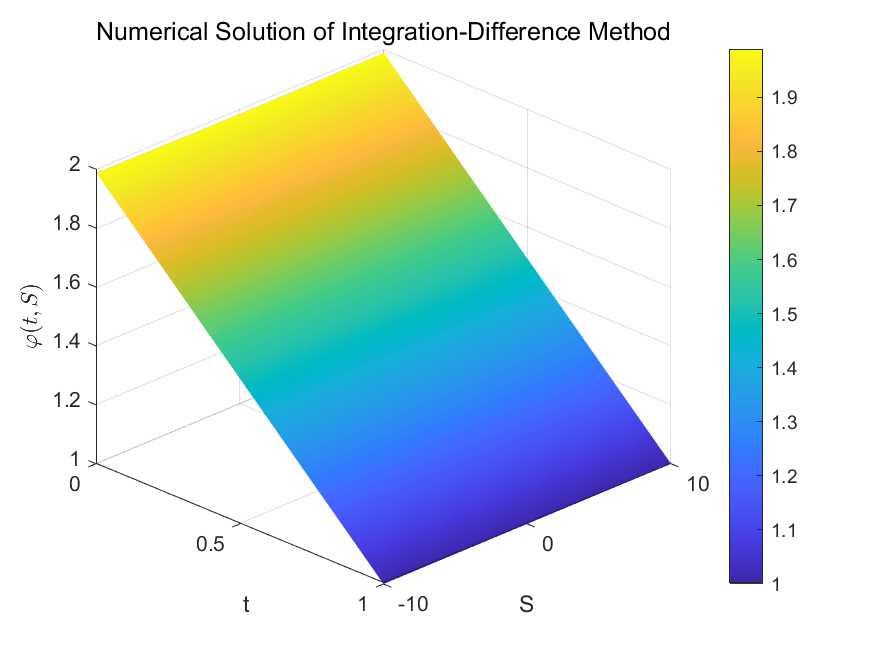}
        \caption{}
        \label{num_idm}
    \end{subfigure}
   \quad
    \begin{subfigure}[b]{0.45\textwidth}
        \centering
        \includegraphics[width=\textwidth]{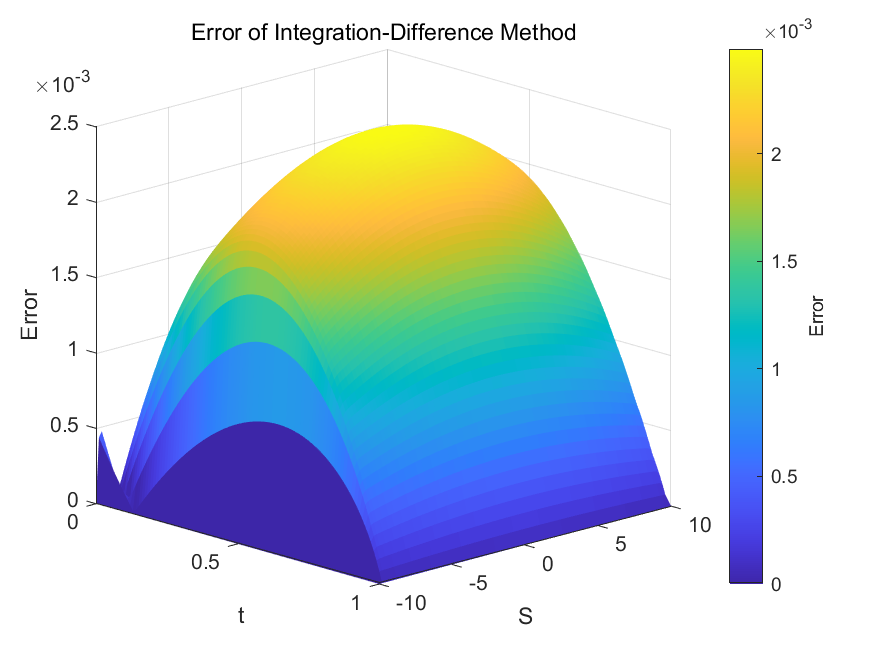}
        \caption{}
        \label{error_fdm}
    \end{subfigure}
    \begin{subfigure}[b]{0.45\textwidth}
        \centering
        \includegraphics[width=\textwidth]{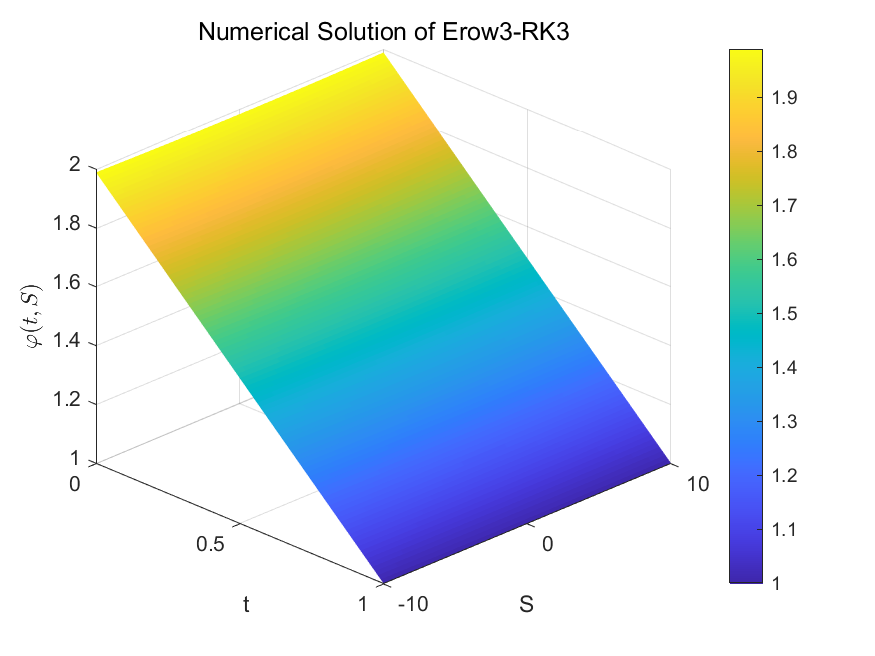}
        \caption{}
        \label{num_Erow3-RK3}
    \end{subfigure}
   \quad
    \begin{subfigure}[b]{0.45\textwidth}
        \centering
        \includegraphics[width=\textwidth]{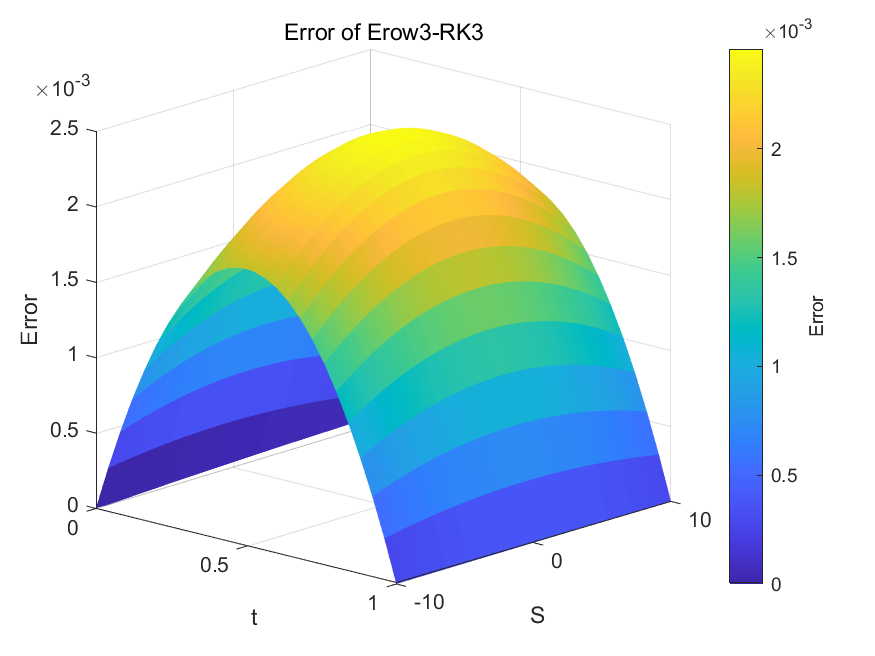}
        \caption{}
        \label{error_Erow3}
    \end{subfigure}
    \caption{Numerical Solutions and Errors}
    \label{f5}
\end{figure}

In Figure \ref{f5}, the panels (\ref{num_idm})-(\ref{error_fdm}) present the numerical solution and error of the integration-difference method, while (\ref{num_Erow3-RK3})-(\ref{error_Erow3}) show those of the Erow3-RK3 Method. We observe that these two methods demonstrate similar error distribution. However, the Errow3-RK3 method achieves a computation time of 438 seconds, resulting in a 41\% reduction compared to the FDM (738 seconds), thereby demonstrating its superior computational efficiency. This efficiency gain is likely due to our variable-separation strategy, which allows Errow3-RK3 to solve only the time-dependent ODE subsystem, whereas the FDM must also resolve the state domain.
 	
\subsection{Application of Erow3-RK3}	
 Schwartz \cite{https://doi.org/10.1111/j.1540-6261.1997.tb02721.x} estimates the parameters of exponential O-U model for oil, copper and gold price. In this subsection, we focus on the optimal control problem (\ref{2.100}) subject to the constraint (\ref{2.1}) in $2$-dimensional case and adopt the parameter estimates for oil obtained from different contracts. Additionally, we introduce a perturbation of 0.01 to the off-diagonal entries of the volatility matrix. The coefficients and intervals are 
\begin{align*}
  &r=0.3, \quad  \gamma=0.5, \quad \rho_{0} = 0.03, \quad \rho=[0.02,0.01]^{\top}, \\
  & \alpha=(0.301,0.428)^{\top}, \quad \mu=(3.093,2.991)^{\top}, 
   \sigma=
\begin{pmatrix}
	0.334 &   0.01  \\
	0.01     & 0.257
\end{pmatrix},\\
&\varrho=
\begin{pmatrix}
	0.002 &   0  \\
	0     & 0.002
\end{pmatrix},\quad t\in[0,0.25], \quad S\in [1,3]\times [1,3].
 \end{align*}
A procedural flowchart is provided in Figure \ref{fp} to illustrate the solution process for the original optimal control problem.

\begin{figure}[h]
    \centering    \includegraphics[width=0.7\textwidth]{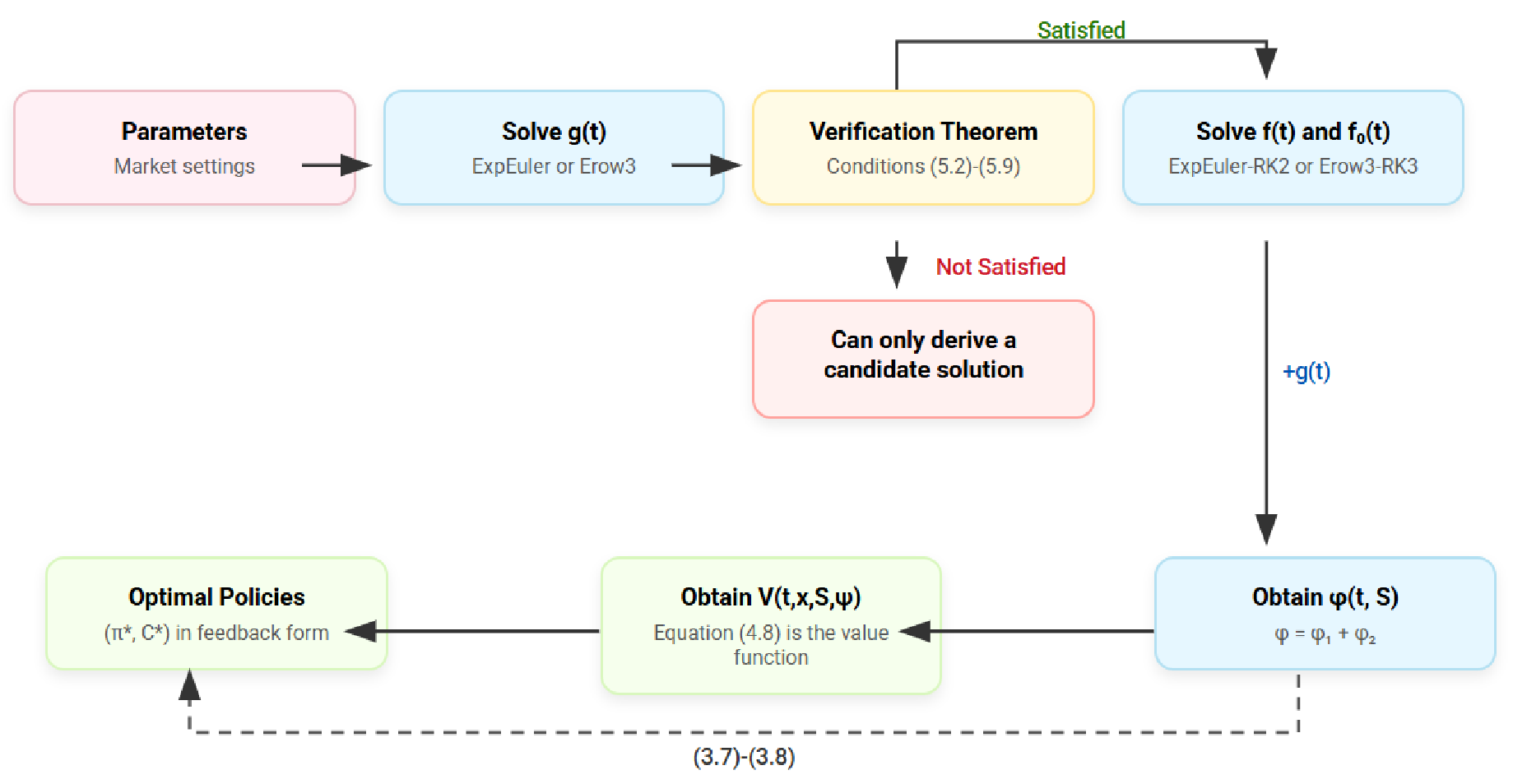}
    \caption{Procedural flowchart}
    \label{fp}
\end{figure}

 Following the procedure outlined in Figure \ref{fp}, we first verify whether conditions (\ref{4.3})-(\ref{4.10}) are satisfied. This is done by reformulating each inequality so that all terms are on the right-hand sides, and then checking for positivity. Using Erow3 to solve for $g(t)$ with 100 time discretization steps, we numerically evaluate the values of the right-hand sides and present the results in Figure \ref{f6} to examine whether they remain positive throughout.
 
\begin{figure}[h]
	\centering
	\begin{subfigure}[b]{0.45\textwidth} 
		\centering
		\includegraphics[width=\textwidth]{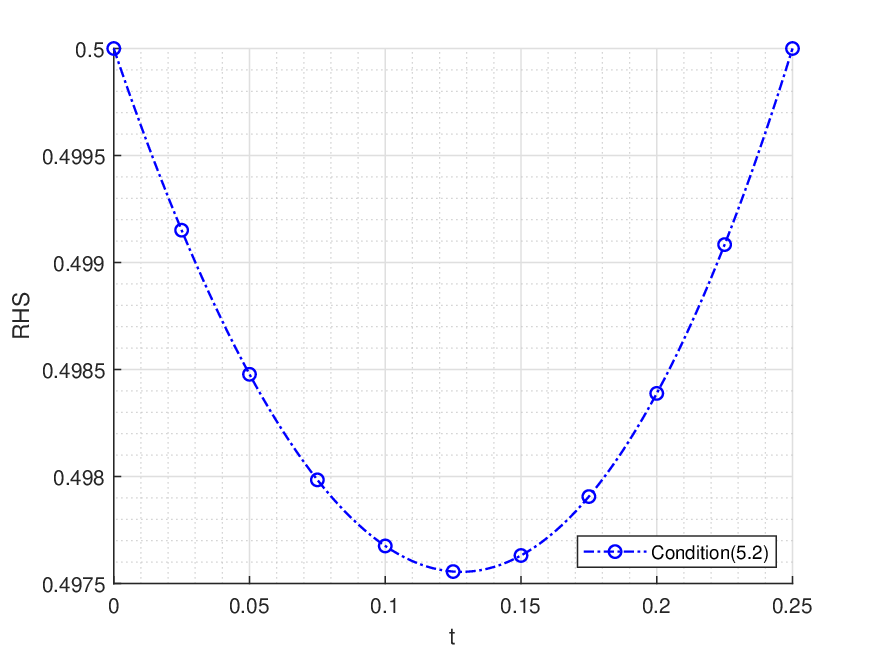} 
	\end{subfigure}
    \quad
	\begin{subfigure}[b]{0.45\textwidth} 
		\centering
		\includegraphics[width=\textwidth]{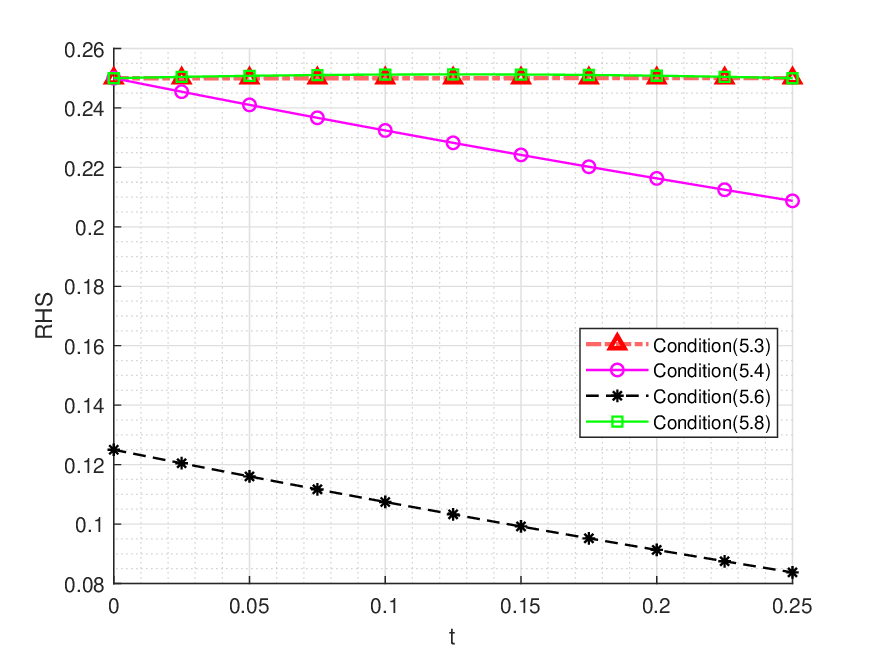} 
	\end{subfigure}
    	\begin{subfigure}[b]{0.45\textwidth} 
		\centering
		\includegraphics[width=\textwidth]{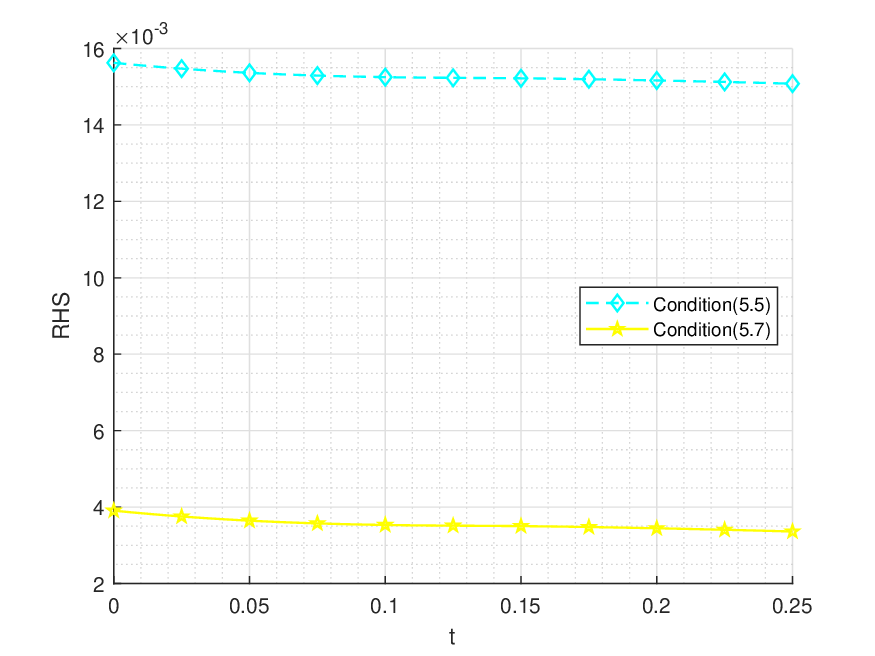} 
	\end{subfigure}
    \quad
    	\begin{subfigure}[b]{0.45\textwidth} 
		\centering
		\includegraphics[width=\textwidth]{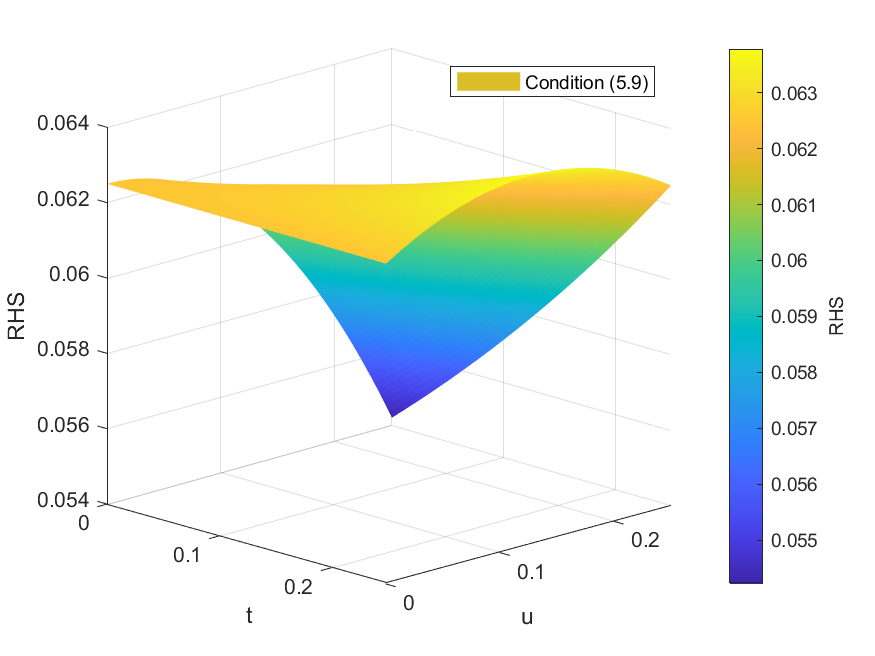} 
	\end{subfigure}
    \caption{Verification Conditions}
	\label{f6}
\end{figure} 

As shown in Figure \ref{f6}, the reformulated right-hand sides of conditions (\ref{4.3})-(\ref{4.10}) are all positive, which confirms that the verification conditions are fully satisfied. Then we employ Erow3-RK3 to solve $\varphi(t,S)$ with 50 time discretization steps. Subsequently, we derive the numerical optimal policy function from (\ref{3.221})-(\ref{3.222}), which is a feedback control. Starting from initial state values, once $X(t_k),S(t_k)$ are obtained, they are substituted into the policy formula to determine the portfolio investment and consumption values. These control inputs are then used to drive the constrained SDE system forward to compute $X(t_{k+1})$ and $S(t_{k+1})$ dynamically. We set the initial state as $S(0) = 2$. Given the computational intractability of determining $x$ according to (\ref{4.13}), we opt for a conservative choice of $x$ large enough to ensure positivity of the resulting wealth processes under all candidate policies. In our case, we set $x = 25$. We apply Euler-Maruyama Method to simulate the state process ${S(t)}$, the optimal wealth process $\{\hat{X}^{*}(t)\}$, the absolute value of the optimal portfolio process $\{|\hat{\pi}^{*}(t)|\}$, and the optimal consumption process $\{\hat{C}^{*}(t)\}$, with all values presented on a logarithmic scale. We generate 50 simulation paths and plot them in the Figure \ref{f7}.

\begin{figure}[htbp]
    \centering
    \begin{subfigure}[b]{0.45\textwidth}
        \centering
        \includegraphics[width=\textwidth]{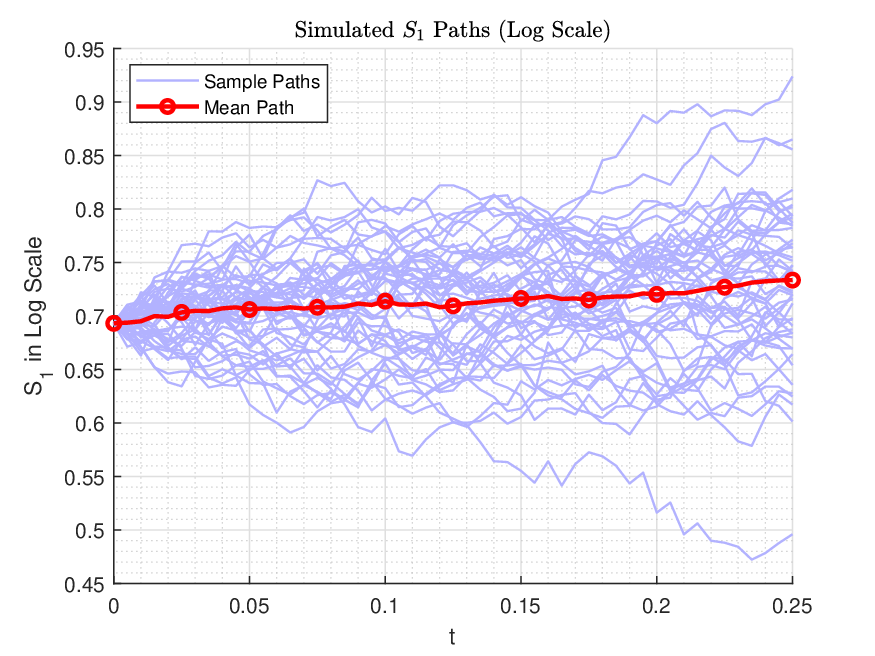}
        \caption{}
        \label{fig:S1}
    \end{subfigure}
    \quad
    \begin{subfigure}[b]{0.45\textwidth}
        \centering
        \includegraphics[width=\textwidth]{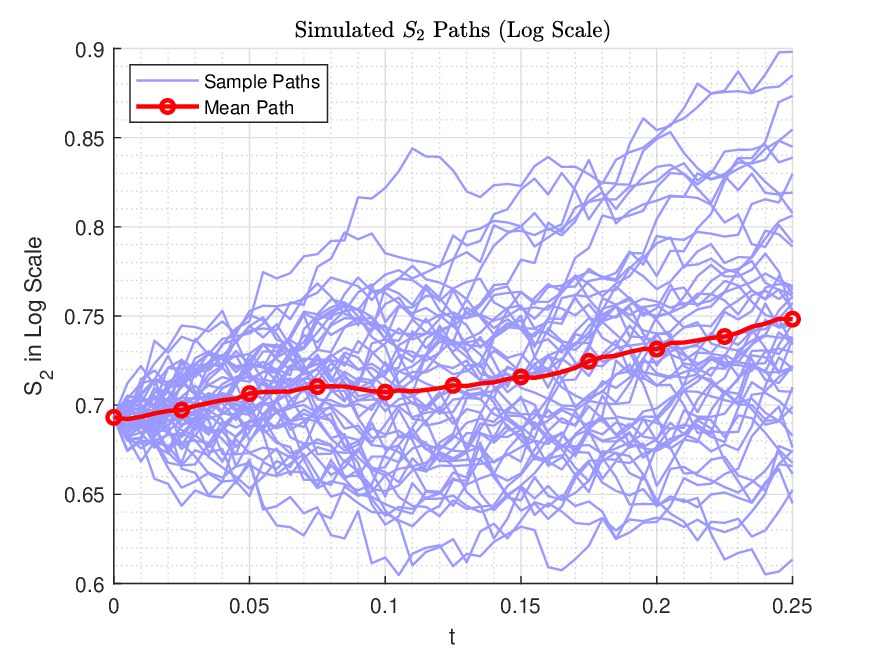}
        \caption{}
        \label{fig:S2}
    \end{subfigure}
    \caption{Simulation Results}
    \label{f7}
\end{figure} 

\begin{figure}[h]
    \ContinuedFloat  
    \centering
    \begin{subfigure}[b]{0.45\textwidth}
        \centering
        \includegraphics[width=\textwidth]{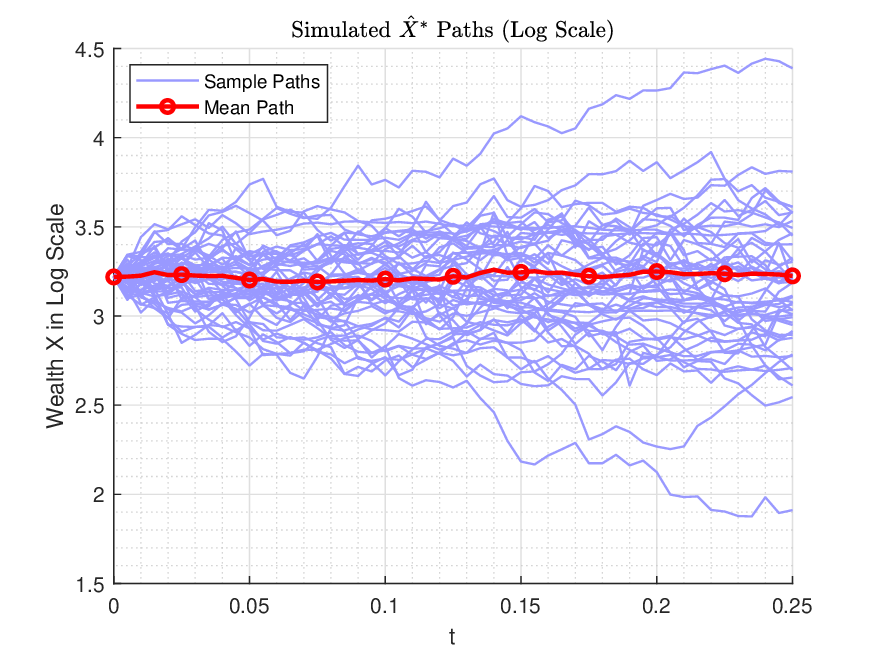}
        \caption{}
        \label{fig:X}
    \end{subfigure}
    \quad
    \begin{subfigure}[b]{0.45\textwidth}
        \centering
        \includegraphics[width=\textwidth]{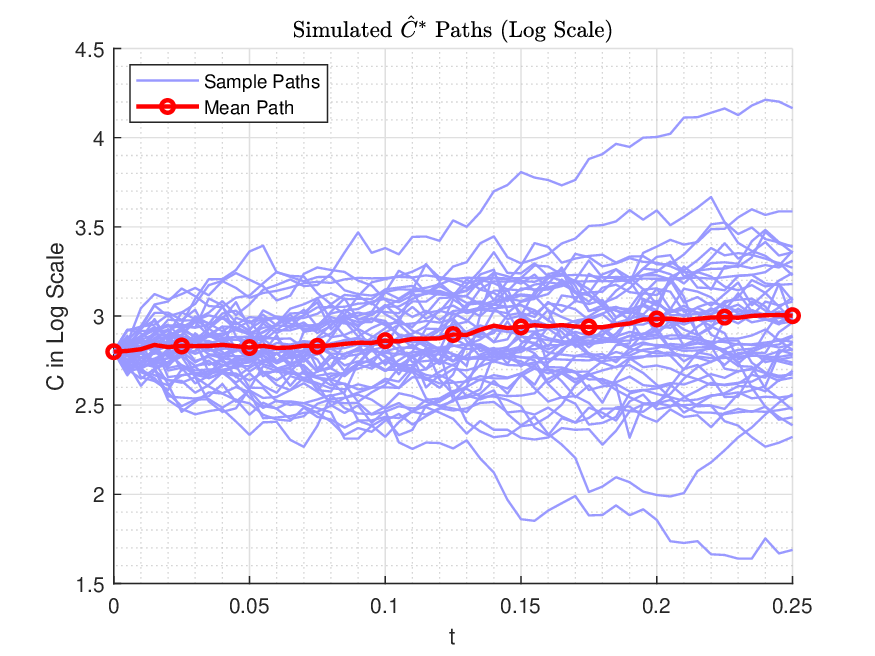}
        \caption{}
        \label{fig:c}
    \end{subfigure}

    \begin{subfigure}[b]{0.45\textwidth}
        \centering
        \includegraphics[width=\textwidth]{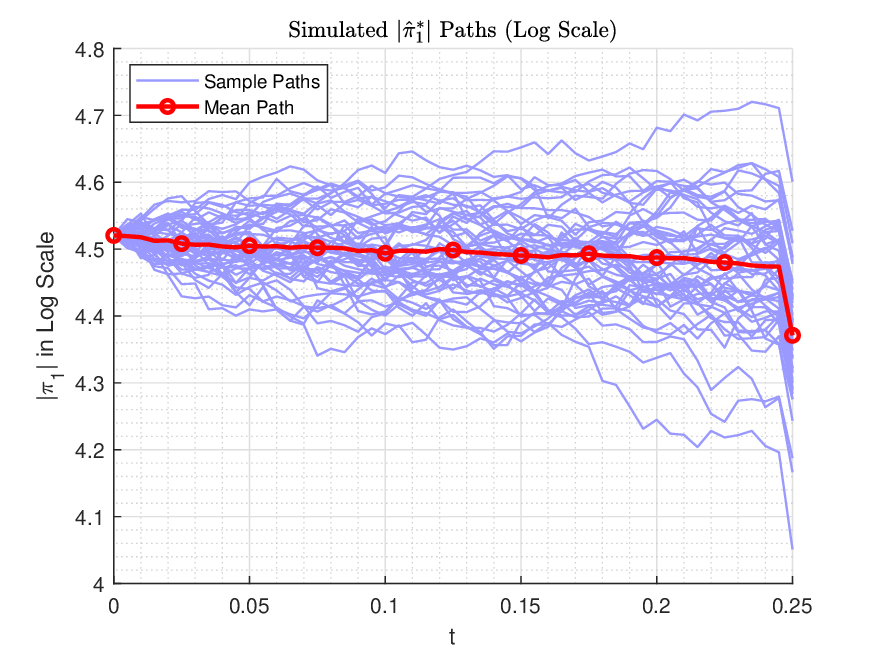}
        \caption{}
        \label{fig:pi1}
    \end{subfigure}
    \quad
    \begin{subfigure}[b]{0.45\textwidth}
        \centering
        \includegraphics[width=\textwidth]{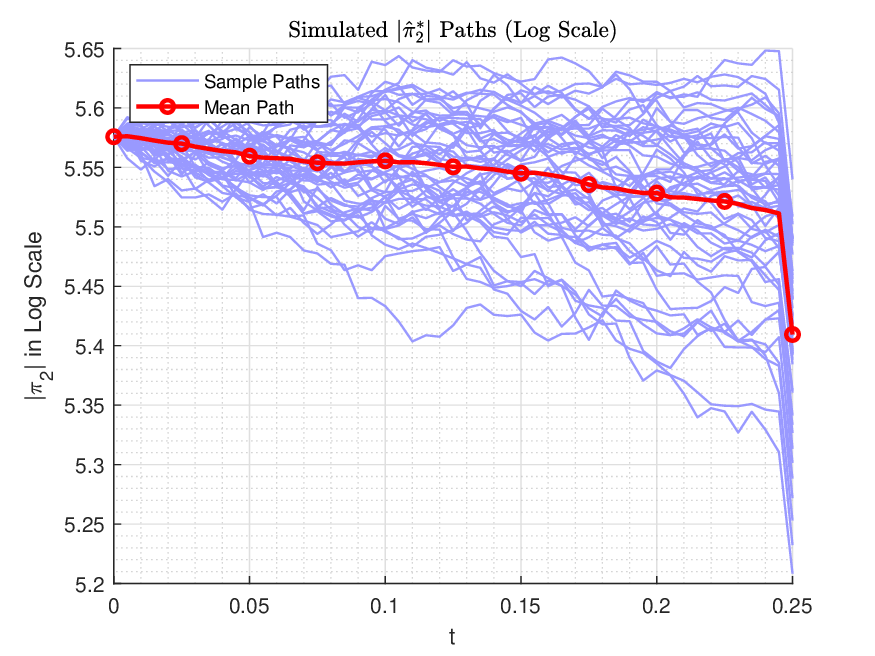}
        \caption{}
        \label{fig:pi2}
    \end{subfigure}
    \caption{Simulation Results}
\end{figure}

From Figure \ref{f7}, panel (\ref{fig:X}) shows that $\hat{X}^{*}$ exhibits considerable fluctuation magnitude while the mean path remains relatively stable. Panels (\ref{fig:pi1}) and (\ref{fig:pi2}) illustrate that both $|\hat{\pi}^{*}_1|$ and $|\hat{\pi}^{*}_2|$ exhibit initial stability, followed by a significant downward trajectory in the latter period. Furthermore, both components of $|\hat{\pi}^{*}|$ exhibit episodes of high leverage at certain times, which may be considered unrealistic in real-world markets. However, since our study is primarily theoretical, such outcomes are acceptable in this context. In practical applications, however, it is crucial to impose constraints on the leverage ratio. 

To validate the optimality of our numerical result, we compute the mean utility across simulated paths to approximate value of obejective functional (\ref{2.100}) and compare it with that obtained from alternative admissible policies. However, computing the utility of every admissible policy is infeasible. We consider several commonly studied investment policies, including the riskless policy \cite{e5a1bb8f-41b7-35c6-95cd-8b366d3e99bc}, the consumptionless policy  \cite{MERTON1971373}, the bondless policy (no risk-free asset) \cite{doi:10.1287/mnsc.2021.3989}, the random policy \cite{10.1007/3-540-36626-1_11}, leverage policies with different leverage ratio \cite{vanRensburg2016}, and our numerical policy.

\begin{table}[!ht]
    \centering
    \resizebox{\textwidth}{!}{
    \begin{tabular}{|l| c| c| c|c|}
        \hline
        Policies & $\pi_1(t,X_t)$ & $\pi_2(t,X_t)$ & $C(t,X_t)$ & Mean Utility \\
        \hline
        Riskless & 0 & 0 & $X_t/2$ & 5.9252 \\
        No Consumption & $X_t/3$ & $X_t/3$ & 0 & 4.6904 \\
        No Consumption (alt.) & $X_t/4$ & $X_t/2$ & 0 & 4.6952 \\
        No Bonds & $X_t/2$ & $X_t/2$ & $X_t/4$ & 5.6821 \\
        No Bonds (alt.) & $2X_t/3$ & $X_t/3$ & $X_t/3$ & 5.7994 \\
        Random & $ \xi_1 \cdot X_t/2$ & $\xi_2 \cdot X_t/2  $ & $\xi_3 \cdot X_t/4$ & 5.3393 \\
        \hline
        Balanced Leverage & $X_t$ & $-X_t/2$ & $X_t/2$ & 5.9368 \\
        Moderate Leverage & $3X_t$ & $-2.5X_t$ & $X_t/3$ & 5.5259 \\
        High Leverage & $-5X_t$ & $5X_t$ & $2X_t/3$ & 5.2690 \\
        Extreme Leverage & $-8X_t$ & $10X_t$ & $X_t/4$ & 3.6834 \\
        \hline
        Our Numerical Policy & $\hat{\pi}^{*}_1(t,X_t)$ & $\hat{\pi}^{*}_2(t,X_t)$ & $\hat{C}^{*}(t,X_t)$ & 6.3022 \\
        \hline
    \end{tabular}
    }
    \caption{Mean Utility Comparison }
    \label{t1}
\end{table}

Table \ref{t1} presents a comparison of mean utilities. The first column lists the investment policies under consideration, the second to fourth columns detail the corresponding $(\pi, C)$ policies, and the final column reports the associated mean utility for each policy. We note that under the random policy, $\xi_1, \xi_2, \xi_3$ are independent and identically distributed uniform random variables on the interval [0,1]. Table \ref{t1} demonstrates that $(\hat{\pi}^{*},\hat{C}^{*})$ outperforms other admissible policies in terms of mean utility.

\section{Conclusions}
    \label{sec7}
	In this paper, we studied the problem of multi-asset portfolio-consumption optimization with exponential O-U stock dynamics and stochastic discounting CRRA utility. We developed hybrid numerical methods to address the computational challenges of high-dimensional stochastic optimal control problems. Leveraging variable separation techniques, we reformulate the HJB equation as an ODE system and develop the ExpEuler-RK2 and Erow3-RK3 schemes for efficient numerical approximation. Most importantly, we derived a rigorous verification theorem with numerically verifiable conditions that guarantee the optimality of our solutions, which distinguishes our approach from existing numerical methods that typically lack theoretical guarantees. Our methods demonstrate significant computational advantages, with Erow3-RK3 achieving 41\% faster computation than conventional grid-based approach, while maintaining third-order convergence. Numerical experiments confirm that the optimal policies computed by our method outperform alternative policies by achieving higher mean utility, validating both the accuracy and practical effectiveness of our methodology.

\section*{Acknowledgements}
This work was supported by the National Key Research and Development Program of China under the grant No. 2020YFA0713602 and the Graduate Innovation Fund of Jilin University under the grant Nos. 2024KC036, 2025CX097 and 2025CX090.

            
			\printbibliography[ title={References}]
			
			

		\end{document}